%% file: Tilings.tex
\documentclass{amsart}
\usepackage{amsmath,enumitem, amssymb, graphicx, epsfig, verbatim, psfrag, color,amscd,stmaryrd} 
\usepackage{tikz}
\usetikzlibrary{arrows}
\DeclareMathAlphabet\mathbfcal{OMS}{cmsy}{b}{n}

\def\endproof{\relax\ifmmode\expandafter\endproofmath\else
  \unskip\nobreak\hfil\penalty50\hskip.75em\hbox{}\nobreak\hfil\bull
  {\parfillskip=0pt \finalhyphendemerits=0 \bigbreak}\fi}
\def\endproofmath$${\eqno\bull$$\bigbreak}
\def\bull{\vbox{\hrule\hbox{\vrule\kern3pt\vbox{\kern6pt}\kern3pt\vrule}\hrule}}

\newcommand\basvec[1]{\mathrm{e}_{#1}}

\newcommand\CDisk{\mathbb D}

\newcommand\ground{\mathfrak k}

\newcommand\gr{\mathrm{gr}}
\newcommand\Pong[2]{{\mathcal P}(#1,#2)}

\newcommand\OneHalf{\frac{1}{2}}
\newcommand\weight{\mathfrak w}

\newcommand\Alg\AlgA
\newcommand\Blg\AlgB

\newcommand\Ainf{{\mathcal A}_\infty}
\newcommand\Ainfty\Ainf

\newcommand\Clg{\mathcal C}
\newcommand\ClgZ{{\mathcal C}_{\Z}}
\newcommand\IdempRing{{\mathcal I}}

\newcommand\Zmod[1]{{\mathbb Z}/{#1}{\mathbb Z}}
\newtheorem{thm}{Theorem}[section]

\newtheorem{cor}[thm]{Corollary}

\newtheorem{lemma}[thm]{Lemma}

\newtheorem{defn}[thm]{Definition}

\newtheorem{remark}[thm]{Remark}

\numberwithin{equation}{section}

\newcommand\Idemp[1]{{\mathbf{I}}_{#1}}

\newcommand{\AlgA}{{\mathcal A}}
\newcommand{\AlgB}{{\mathcal B}}

\newcommand{\C}{\mathbb C} \newcommand{\Z}{\mathbb Z}   \newcommand{\R}{\mathbb R}

\newcommand\Field{\mathbb F}

\DeclareMathOperator{\Id}{Id}

\begin{document}
\title{Planar graphs deformations of bordered
  knot algebras}

\begin{abstract}
  In an earlier paper, we introduced ``bordered knot algebras'', which
  are graded algebras indexed by a pair of integers $(m,k)$.  In a
  subsequent paper, we introduced a two-parameter family of
  differential graded algebra, the ``pong algebras'', and identified
  their homology with the bordered knot algebras, and characterized
  the induced $A_\infty$ structure on the homology. The aim of the
  present paper is to give an explicit, combinatorial model for this
  $A_\infty$ structure on the bordered knot algebras, and a further
  weighted deformation of this structure, in the case where $k=1$.
\end{abstract}

\author[Peter S. Ozsv\'ath]{Peter Ozsv\'ath}
\thanks {PSO was partially supported by NSF grant number DMS-2104536, and the Simons Grant {\em New structures in low-dimensional topology}.}
\address {Department of Mathematics, Princeton University\\ Princeton, New Jersey 08544} 
\email {petero@math.princeton.edu}

\author[Zolt{\'a}n Szab{\'o}]{Zolt{\'a}n Szab{\'o}}
\thanks{ZSz was supported by NSF grant number DMS-1904628
  and the Simons Grant {\em New structures in low-dimensional topology}.}
\address{Department of Mathematics, Princeton University\\ Princeton, New Jersey 08544}
\email {szabo@math.princeton.edu}

\newcommand\eVec[1]{{\mathrm e}_{#1}}
\newcommand\BordAlg[2]{\Clg(#1,#2)}
\input{intro}
\input{Cm1}

\input{algebra}
\input{Patterns}

\input{geomrel}

\input{signs}

\bibliographystyle{plain}
\bibliography{biblio}

\end{document}

%% file: intro.tex
\maketitle
\section{Introduction}

In~\cite{BorderedKnots}, we described certain algebras $\Clg(m,k)$,
which are related to knot Floer homology. Moreover
(cf.~\cite[Theorem~\ref{Pong:thm:HomologyPongAinf}]{Pong}; see also
Theorem~\ref{thm:UniqueCm} below), we showed that $\Clg(m,k)$ can be
given an $\Ainf$ structure which is uniquely characterized by its
grading and having a non-trivial action.

In this paper, we give an explicit, combinatorial description of this
$\Ainf$ structure, in the special case where $k=1$, to give an
explicit $\Ainf$ algebra $\Clg^0(m,1)$.  Furthermore, we give a
combinatorial description of certain weighted deformation ${\mathbfcal
C}(m,1)$ of $\Clg^0(m,1)$, in the sense of~\cite{AbstractDiag}.

The combinatorial description counts certain planar graphs; as such it
is closely related to the construction of the weighted algebra
associated to the torus from~\cite{TorusAlg}. Indeed, the verification of
the structure relation follows closely the corresponding verification
from~\cite[Section~3]{TorusAlg}.

This paper is organized as follows. In Section~\ref{sec:Cm1} we review
the construction of the bordered algebra ${\mathcal C}(m,1)$.  In
Section~\ref{sec:WtAlg} we recall the notion of weighted $A_\infty$
algebras, and set our grading conventions.  In
Section~\ref{sec:Tilings}, we define a class of decorated planar
graphs, and use these decorated graphs to define both the $A_\infty$
deformation of ${\mathcal C}(m,1)$ and its further weighted
deformation.  In Section~\ref{sec:Relation} we verify that the
algebras defined above satisfy the stated $A_\infty$ relation. In
Section~\ref{sec:Signs} we explain how to introduce signs into the
above discussion.

The constructions from this paper are related with constructions in
the wrapped Fukaya category; see~\cite[Theorem~7.26]{TorusAlg}
and~\cite{WrapPong}. In a little more detail, recall that in~\cite{Pong},
we constructed a differential graded algebra, the
{\em pong algebra}

The following is a special case of~\cite[Theorem~1.1]{Pong} (with $k=0$):

\begin{thm}
  The homology of the differential $\Pong{m}{m-1}$, with its induced $\Ainfty$
  structure is isomorphic to the $\Ainfty$ structure $\Clg(m,1)[t]$
  with the gradings and non-trivial $\mu_{2m-2}$ operation constructed
  in this paper.
\end{thm}

Moreover,~\cite[Theorem~1.1]{Pong}, identifies $\Pong{m}{m-1}$ with
the endomorphism of a collection of objects Lagrangians associated to
parallel, vertical lines in a wrapped Fukaya category of
$(m-1)$-fold symmetric product of $\C$.  Thus, we can view the computations
here as giving a more efficient shorthand for computing the operations
in that wrapped Fukaya category.

There is another, Koszul dual perspective, though. We can consider
$\C$ with $m$ punctures, and consider arcs connecting those
punctures. The algebra $\Clg(m,1)$ can be identified with a wrapped
Fukaya category in $\C$ (i.e. the first symmetric product) relative to
those $m$ punctures; see~\cite{LekiliPolishchuk}. We will return to
this in future work~\cite{Koszul}.

It is interesting to compare this work with other Fukaya categorical
approaches to bordered algebras; see for
example~\cite{KotelskiyWatsonZibrowius,LaudaLicataManion,EllisPetkovaVertesi,Zibrowius}.

{\bf Acknowledgements:} We wish to thank Robert Lipshitz, Dylan
Thurston, and Andrew Manion for interesting discussions leading up to
this work.  The collaboration of the first author with Lipshtiz and
Thurston (specifically~\cite{TorusAlg}) served as an inspiration for
this work.

%% file: Cm1.tex
\section{The bordered algebra ${\mathcal C}(m,1)$}
\label{sec:Cm1}

We consider bordered algebras $\BordAlg{m}{k}$ from~\cite{BorderedKnots} with $k=1$,
following notation from~\cite[Section~3.2]{HolKnot}; 
  also~\cite[Section~\ref{Pong:sec:Clg}]{Pong}.

We give the following simple, concrete description of $\BordAlg{m}{1}$ for $m>2$ over $\Field=\Zmod{2}$.
Consider the quiver $Q$ with vertices labeled $\{[1],\dots,[m-1]\}$
and whose edges are $L_i\colon [i]\to [i-1]$,
$R_i\colon [i-1]\to [i]$ with $2\leq i\leq m-2$; and 
$U_1\colon [1]\to [1]$, $U_m \colon [m-1]\to [m-1]$.
(See Figure~\ref{fig:Quiver} for an illustration when $m=4$.)
Consider its path algebra $P$ for $Q$, divided out by the relations:
\[ L_i L_{i-1}=0, \qquad R_{i-1} R_{i}=0,\qquad L_2 U_1 = U_1 R_2=R_{m-1} U_m=U_m L_{m-1}=0,\]
with $3\leq i\leq m-1$.
The vertices $[x]$ in $Q$ correspond to idempotents $I_{x}$ in the path algebra.
We find it convenient to write 
$U_i=R_i L_i + L_i R_i$ for $i=2,\dots, m-1$.

\begin{figure}[ht]
\input{Quiver.pstex_t}
\caption{\label{fig:Quiver} {\bf{Quiver for $\Clg(4,1)$.}}}
\end{figure}
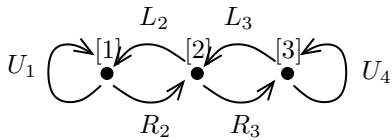

In the degenerate case where $m=2$, we take $\Clg(2,1)=\Field[U_1,U_2]/U_1 U_2$.

Let $\IdempRing(m,1)\subset \Clg(m,1)$ denote the subring of
idempotents corresponding to the constant paths; and let
$\Clg_+(m,1)\subset \Clg(m,1)$ denote the $\IdempRing(m,1)$-bimodule
generated by non-constant paths.

\begin{defn}
  \label{def:Pure}
 An algebra element in $\Clg(m,1)$
corresponding to a non-constant path is called a {\em pure algebra element}.
\end{defn}
For example, in $\Clg(4,1)$, corresponding to the quiver from
Figure~\ref{fig:Quiver}, the elements $U_1$ and $R_2\cdot L_2$ are
pure algebra elements.  Note that $U_2=R_2\cdot L_2+L_2\cdot R_2$ is not a pure algebra element.

The algebra $\BordAlg{m}{1}$ is equipped with a grading denoted $|\cdot |$
with values in $\OneHalf \Z^m$, called the {\em shadow grading}, characterized by the following properties:
\begin{align*}
  |\iota|&= 0 \\
  |L_i|&=|R_i| = \OneHalf \eVec{i} \\
  |U_1|&= \eVec{1} \\
  |U_m|&= \eVec{m},
\end{align*}
where $\{\eVec{i}\}_{i=1}^m$ is the standard basis for $\Z^m$, and $\iota$
is any idempotent. In earlier papers, this grading was called a
``weight grading'' and it was denoted ${\mathfrak w}$. We do not use
this language presently, so as to separate from the notion of
``weighted algebras'', as in Section~\ref{sec:WtAlg}.

We consider now $\Ainfty$ deformations of $\BordAlg{m}{1}$, in the
following sense.  Consider the bigraded algebra $\BordAlg{m}{1}[t]$,
where the shadow grading is extended so that
$|t|=\sum_{i=1}^m
\eVec{i}$; and an additional $\Z$-valued grading $\gr$, with
$\gr(\BordAlg{m}{1})=0$ and $\gr(t)=2m-4$.

We consider graded $\Ainfty$ structures on  $\BordAlg{m}{1}[t]$, satisfying the grading requirements:
  \begin{align}
    \label{eq:GradedHypotheses}
    |\mu_{\ell}(a_1,\dots,a_\ell))|&=\sum_{i=1}^{\ell}|a_i| \\
    \gr(\mu_{\ell}(a_1,\dots,a_\ell))&=\ell-2+ \sum_{i=1}^{\ell}\gr(a_i). \nonumber
  \end{align}

We have the following:

\begin{thm}
  \label{thm:UniqueCm} [Theorem~\ref{Pong:thm:CharacterizeActions} of~\cite{Pong}]
  Up to quasi-isomorphism, there is a unique graded $\Ainf$ structure on 
  $\BordAlg{m}{1}[t]$ (with $m>2$) extending the natural algebra structure on $\BordAlg{m}{1}$,
  with non-trivial $\mu_{2m-2}$-operation.
\end{thm}

\begin{remark}
  The degenerate case $m=2$ holds as well. 
  In that case, we are deforming $\Clg(2,1)\cong\Field[U_1,U_2]/U_1 U_2$
  to the algebra $\Field[U_1,U_2]$, where $t=U_1 U_2$. (For example, we introduce operations $\mu_2(U_1,U_2)=\mu_2(U_2,U_1)=t$.)
\end{remark}

Theorem~\ref{thm:UniqueCm} is a special case of~\cite[Theorem~\ref{Pong:thm:CharacterizeActions}]{Pong}, where uniqueness is
actually proved for the two-parameter family $\BordAlg{m}{k}$; compare~\cite[Theorem~5.44]{TorusAlg}.

%% file: Quiver.pstex_t
\begin{picture}(0,0)%
\includegraphics{Quiver.pstex}%
\end{picture}%
\setlength{\unitlength}{1657sp}%
\begingroup\makeatletter\ifx\SetFigFont\undefined%
\gdef\SetFigFont#1#2#3#4#5{%
  \reset@font\fontsize{#1}{#2pt}%
  \fontfamily{#3}\fontseries{#4}\fontshape{#5}%
  \selectfont}%
\fi\endgroup%
\begin{picture}(5340,2029)(1156,-5579)
\put(3826,-4336){\makebox(0,0)[lb]{\smash{{\SetFigFont{10}{12.0}{\rmdefault}{\mddefault}{\updefault}{\color[rgb]{0,0,0}$[2]$}%
}}}}
\put(5176,-4336){\makebox(0,0)[lb]{\smash{{\SetFigFont{10}{12.0}{\rmdefault}{\mddefault}{\updefault}{\color[rgb]{0,0,0}$[3]$}%
}}}}
\put(2476,-4336){\makebox(0,0)[lb]{\smash{{\SetFigFont{10}{12.0}{\rmdefault}{\mddefault}{\updefault}{\color[rgb]{0,0,0}$[1]$}%
}}}}
\put(1171,-4561){\makebox(0,0)[lb]{\smash{{\SetFigFont{10}{12.0}{\rmdefault}{\mddefault}{\updefault}{\color[rgb]{0,0,0}$U_1$}%
}}}}
\put(4411,-3841){\makebox(0,0)[lb]{\smash{{\SetFigFont{10}{12.0}{\rmdefault}{\mddefault}{\updefault}{\color[rgb]{0,0,0}$L_3$}%
}}}}
\put(4501,-5461){\makebox(0,0)[lb]{\smash{{\SetFigFont{10}{12.0}{\rmdefault}{\mddefault}{\updefault}{\color[rgb]{0,0,0}$R_3$}%
}}}}
\put(3151,-5461){\makebox(0,0)[lb]{\smash{{\SetFigFont{10}{12.0}{\rmdefault}{\mddefault}{\updefault}{\color[rgb]{0,0,0}$R_2$}%
}}}}
\put(3151,-3841){\makebox(0,0)[lb]{\smash{{\SetFigFont{10}{12.0}{\rmdefault}{\mddefault}{\updefault}{\color[rgb]{0,0,0}$L_2$}%
}}}}
\put(6481,-4651){\makebox(0,0)[lb]{\smash{{\SetFigFont{10}{12.0}{\rmdefault}{\mddefault}{\updefault}{\color[rgb]{0,0,0}$U_4$}%
}}}}
\end{picture}%

%% file: algebra.tex
\section{Weighted $A_\infty$ algebras}
\label{sec:WtAlg}

Throughout this paper, we will use the language of weighted $A_\infty$
algebras from~\cite{AbstractDiag}, recalling the basic definitions
here. For simplicity, these will be over a ground ring $\ground$ with
characteristic $2$. We return to the case of signs in
Section~\ref{sec:Signs}. The definition of a weighted $A_\infty$
algebra is as follows.

\begin{defn}
Fix a real $n$-dimensional vector space $W$ with a preferred basis
$\{\basvec{i}\}_{i=1}^n$, and let $W_{\geq 0}$ denote the set of vectors that
can be written in the form $\sum_i w_i \basvec{i}$ where $w_i\in \R^{\geq
  0}$.  Fix an algebra $\ground$ with characteristic $2$.
 A {\em $W$-weighted $A_\infty$ algebra
  over $\ground$} is the following object is the following data:
\begin{itemize}
\item a graded $\ground$-bimodule $A$
\item for each $w\in W$, a linear map
  of $\ground$-bimodules
  \[ \mu^{w}_n \colon \overbrace{A\otimes \dots\otimes A}^n \to A,\]
  where the tensor it taken over $\ground$. Moreover, $\mu^w_n$ is
  of degree $n-2+2\sum w_i$; i.e. if $a_1,\dots,a_n$ are homogeneous,
  then $\mu^{w}_n(a_1,\dots,a_n)$ is homogeneous of degree
  \[ \gr(\mu^{w}_n(a_1,\dots,a_n))=\left(\sum \gr(a_i)\right)+n-2 + 2\left(\sum_i w_i\right).\]
  In particular, when the sequence of algebra elements is empty, we obtain
  elements $\mu^w_0\in A$ of degree $-2+2\sum w_i$.
  The vector $w\in W$ is called the {\em weight} of the operation $\mu^{w}_n$.
\item The element $\mu^{\vec{0}}_0=0$.
\end{itemize}
These operations are further required to satisfy the following structure relation  for all 
sequences $(a_1,\dots,a_n)$ in $A$
and vectors $w\in W_{\geq 0}$:
\[ \sum_{1\leq i\leq j\leq n+1,u+v=w} \mu^u_{n-j+i+1}(a_1,\dots,a_{i-1},
\mu^v_{j-i}(a_i,\dots,a_{j-1}),a_j,\dots,a_n)=0.\]
\end{defn}

We have suppressed signs in the above formula, as we are presently
working in characteristic two. See Equation~\eqref{eq:WeightedAinf}
for the sign refined version. 

\begin{remark}
  When $W$ is one-dimensional, this coincides with the definition of
  weighted algebras~\cite[Definition~4.1]{AbstractDiag}, with
  $\kappa=-2$.  Weighted $\Ainfty$ algebras are equivalent to curved
  $\Ainfty$ algebras over formal power series, with a constraint on
  the curvature; compare~\cite[Chapter~3]{FOOO}.
\end{remark}

Given a weighted $A_\infty$ algebra, we can specialize to the
operations $\{\mu^{\vec{0}}_n\}_{n=1}^{\infty}$, i.e. those with
vanishing weight.  This restriction is an (ordinary) $A_\infty$
algebra, $(A,\mu^{\vec{0}})$. We say that $(A,\mu)$ is a {\em weighted
  deformation} of $(A,\mu^{\vec{0}})$. Moreover, we will typically
abbreviate $\vec{0}$ by $0$.

\begin{defn}
  A weighted $\Ainfty$ algebra is called {\em unital} if 
  there is a distinguished $1\in A$ with the property that
  \[ \mu^0_2(a,1)=\mu^0_2(1,a)=a \]
  and 
  for any sequence of algebra elements $(a_1,\dots,a_n)$ with some $a_i=1$
  and $(n,\vec{w})\neq (2,0)$, we have that 
  $\mu^{\vec{w}}_n(a_1,\dots,a_n)=0$. 
\end{defn}

(Note that this condition is often called ``strictly unital'' in the
literature, to leave room for other types of unitality.)

The algebra $\Clg(m,1)$ is unital, where $1$ is the sum of all the idempotents
corresponding to all the constant paths; i.e.
$1=\Idemp{1}+\dots+\Idemp{m-1}$.

%% file: Patterns.tex
\section{Tiling patterns}
\label{sec:Tilings}

Fix an integer $m\geq 3$. We will describe here the tiling patterns used for
the definition of the deformation of $\Clg(m,1)$ studied in~\cite{Pong}.

The definitions here are parallel to the material from~\cite[Section 3.2]{TorusAlg}.

\begin{defn}
  A {\em planar, rooted graphs} is a graph
  $\Gamma$, 
  equipped with an embedding into the disk $\CDisk$, so that
  $\Gamma\cap\partial\CDisk$ consists of the leaves of $\Gamma$, one of
  which is distinguished and called the {\em root}.
\end{defn}

We consider planar, rooted graphs, each of whose nodes have degree
$2m-2$, satisfying further conditions, formulated below, and illustrated
in Figure~\ref{fig:NodeLabel}.

Around each vertex $v$, orient the edges so that $m-1$ consecutive
edges (which are considered consecutive with respect to the circular
ordering on the edges containing $v$) are oriented into $v$, and the
further $m-1$ consecutive edges are ordered out of $v$. Moreover, we
label the edges into $v$ by integers $1,\dots,m-1$, with respect to
the counter-clockwise ordering; and we label the edges out of $v$
$1,\dots,m-1$ with respect to the clockwise ordering.

\begin{figure}[ht]
\input{Node.pstex_t}
\caption{\label{fig:NodeLabel} {\bf{Labels around each node.}}
We have drawn here the decorations around each node in the case where $m=4$.}
\end{figure}
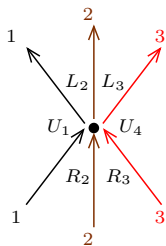

A {\em compatibly decorated graph} is a rooted, planar graph all of
whose edges are oriented and labeled by integers in
$1,\dots,m-1$; and the orientations and labelings are consistent with
the above local conventions about each internal vertex $v$.
Clearly, the labels are determined by the orientations; but not all oriented
graphs can be labeled to satisfy the compatibility conditions; see Figure~\ref{fig:Unlabelable}

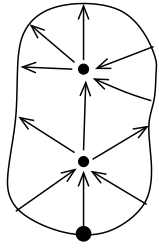
\begin{figure}[ht]
\input{Unlabelable.pstex_t}
\caption{\label{fig:Unlabelable} {\bf{This oriented graph cannot be labeled compatibly.}}}
\end{figure}

For a compatibly decorated graph, there are $2m-2$ sectors around each
internal vertex $v$. We label these sectors by letters
$U_1,L_2,\dots,L_{m-1},U_m,R_{m-1},\dots,R_2$ with respect to the
clockwise ordering of the edges around $v$. Indeed, we label so that
\begin{itemize}
\item  $L_i$  is the sector between 
  the outgoing edge labeled $i-1$ and the outgoing edge labeled $i$.
\item 
  $R_i$ is the sector between the incoming edge labeled $i-1$ and the
  incoming edge labeled $i$.
\item $U_1$ labels the sector between the incoming edge labeled $1$ and the outgoing edge labeled $1$.
\item $U_m$ labels the sector between the outgoing edge labeled $m-1$ and the incoming edge labeled $m-1$.
\end{itemize}

\begin{defn}\label{def:CenteredTilingGraph}
A {\em centered tiling graph} is a compatibly decorated graph satisfying the further two properties:
\begin{itemize}
\item The graph is connected.
\item All cycles have length one.  (This means that if we have a
  sequence of distinct edges $e_1,\dots,e_k$ in $\Gamma$ so that the edge $e_i$
  goes from the vertex $v_i$ to the vertex $v_{i+1}$, then $v_1\neq
  v_{k+1}$, except in the special case where $k=1$.)
\end{itemize}
\end{defn}

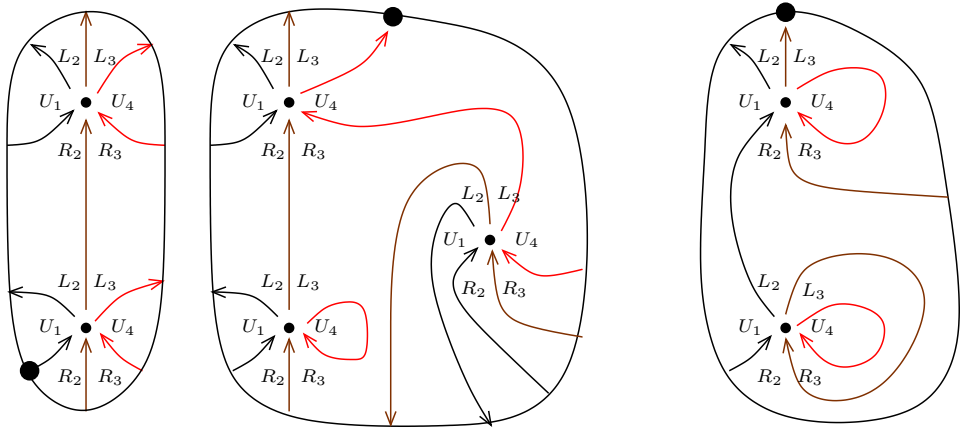
\begin{figure}[ht]
\input{TilingExamples.pstex_t}
\caption{\label{fig:TilingEx} {\bf{Examples of tiling graphs.}}}
\end{figure}

\subsection{Actions}

Each tiling graph $\Gamma$ has a corresponding algebra sequence and weight.

\begin{defn}
  \label{def:AlgebraSequence}
  The {\em algebra sequence} $\vec{a}(\Gamma)$ 
  is obtained as follows. The boundary of $\Gamma$
  divides the boundary of the disk into intervals. For each interval $I$
  there is a unique connected component $X(I)$ of
  $\Delta\setminus\Gamma$. The algebra element associated to $I$,
  $a(I)$, is obtained by multiplying together the algebra elements
  associated to the sectors appearing in $X(I)$. Order the intervals
  $I_1,\dots,I_n$ counterclockwise around the boundary, starting at the
  root vertex.
\end{defn}

The weight vector associated to $\Gamma$, $\vec{w}(\Gamma)$ is
obtained as follows.  Note that $\Gamma$ divides $D$ into $k$
components.  Each component that does has one of the following types:
\begin{itemize}
\item It is a monogon containing $U_1$ (Type $1$)
\item It is a monogon containing $U_m$ (Type $m$) 
\item It is a bigon containing $R_i$ and $L_i$ (Type $i$).
\end{itemize}
The weight vector  $\vec{w}=(w_1,\dots,w_m)$
has components where $w_i$ counts the number of components of type $i$.

The graph $\Gamma$ also determines a idempotent $\iota(\Gamma)=\Idemp{[k]}$, where $k$ is 
the number of the root edge. The graph $\Gamma$ also determines an exponent
$d$ which is  the number of internal vertices.

We add further actions 
\begin{equation}
  \label{eq:AddedMu0s}
  \mu^{\eVec{i}}_0=U_i,
\end{equation} where
$\eVec{i}$ is the standard $i^{th}$ basis vector in the
weight space $\Z^m$. 

Note that there are no further $\mu^{\vec{w}}_0$ actions; there are
also no $\mu^{\vec{w}}_1$ actions. Moreover $\mu^0_2$ is induced from
the underlying algebra, but there are new weighted $\mu^{\vec{w}}_2$
actions with non-zero weight vector. For example, in
Figure~\ref{fig:WeightedM2}, we have illustrated a tree that shows
\[ \mu_2^{2\eVec{3}}(U_1^2,U_2^2)=t^2\Idemp{[1]}.\]

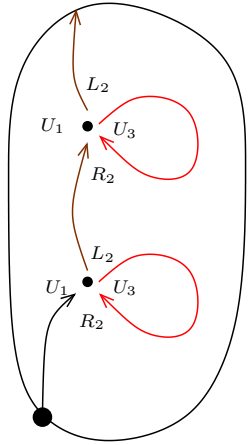
\begin{figure}[ht]
\input{WeightedM2.pstex_t}
\caption{\label{fig:WeightedM2} {\bf{Weighted $\mu_2$ tree.}}}
\end{figure}

The following observation will be useful in the future: 

\begin{lemma}
  \label{lem:GradedOperations}
  Given $\Gamma$, we have that
  \[ m\cdot d(\Gamma)=\sum_{i=1}^n|a_i|+\sum_i w_i\]
  where $(w_1,\dots,w_m)=\weight(\Gamma)$ and $(a_1,\dots,a_n)=\vec{a}(\Gamma)$.
\end{lemma}

\begin{proof}
  The sum of $|b_i|$ around each node equals $m$.
  These algebra elements either contribute to $\sum|a_i|$ 
  if the corresponding region touches the boundary circle;
  otherwise, they contribute to $\sum_i w_i$.
\end{proof}

Define maps
\[ c^{\vec{w}}_n\colon \Clg_+^{\otimes n}\to \IdempRing[t]\]
(where the tensor product on $\Clg_+$ is understood over the ground ring $\IdempRing{m,1}$)
by the formula
\[ c^{\vec{w}}_n(a_1,\dots,a_n)=\sum_{\{\Gamma\big|w(\Gamma)=\vec{w},{\vec a}(\Gamma)=(a_1,\dots,a_n)\}} \iota(\Gamma) \cdot t^{d(\Gamma)}.\]

For example, the 
pictures in Figure~\ref{fig:TilingEx} correspond to 
\begin{align*}
  c^{0}_{10}(R_2,R_3,U_4,U_3,U_4,L_3,L_2,U_1,U_2,U_1) =t^2 \Idemp{[1]}\\
  c^{\eVec{4}}_{12}(L_3,L_2,U_1,U_2,U_1,R_2,U_3^2,L_2, U_1,R_2,R_3,U_4^2)=t^3\Idemp{[3]} \\
  c^{\eVec{3}+2\eVec{4}}_4(L_2,U_1^2,R_2 U_2,U_3)=t^2\Idemp{[2]}.
\end{align*}
Note that in the above expressions, we are using algebra elements,
such as $U_3$, which are not pure, in the sense of
Definition~\ref{def:Pure}. Rather, $U_3=R_3L_3+L_3R_3$ 
is a sum of two pure algebra
elements, in two differing idempotents. However, 
the initial idempotent is determined in the above expressions.

Next, define
\[
  \mu^{\vec{w}}_n\colon \Clg_+^{\otimes n}\to \Clg[t] 
\]
by 
\begin{align*}\mu^{\vec{w}}_n(a_1,\dots,a_n)=& c^{\vec{w}}_n(a_1,\dots,a_n) \\
& +
\sum_{a_1=a\cdot a_1'} a\cdot c^{\vec{w}}_n(a_1',a_2,\dots,a_n) \\
& +
\sum_{a_n=a_n'\cdot a}  c^{\vec{w}}_n(a_1,\dots,a_{n-1},a_n')\cdot a.
\end{align*}
Terms of the first kind are called {\em centered}; those of the second type are called {\em left-extended}; and those of the third kind are called
called {\em right-extended}.
Extend the above map to
\begin{equation}
  \label{eq:WeightedAinftyAlgebra}
  \mu^{\vec{w}}_n\colon (\Clg[t])^{\otimes n}\to \Clg[t] 
\end{equation}
subject to the following properties:
\begin{itemize}
\item $\mu^{\vec{w}}_n$ is $t$-equivariant
  (in particular, we can view $t$ as lying in the base ring over which
  we are taking the tensor product).
\item 
  $\mu^{\vec{w}}_n(a_1,\dots,a_m)=0$ if any element $a_i\in\Clg[t]$ is an idempodent, provided that $(\vec{w},n)\neq (0,2)$.
\end{itemize}

The input algebra sequence to an extended graph is defined similar to
the centered case (Definition~\ref{def:AlgebraSequence}), bearing in
mind that for the extending vertices, the corresponding algebra
elements are seen only from one side. Similarly, the output algebra element
is the product of all of the algebra elements on the extending sequence,
times $t^{d(\Gamma)}$.

\begin{lemma}
  \label{lem:NonMultiplyable}
  Fix $(\vec{w},n)\neq (0,2)$ and a sequence $(a_1,\dots,a_n)$ of pure
  algebra elements in $\Clg_+$.  If
  $\mu^{\vec{w}}_n(a_1,\dots,a_n)\neq 0$, then for each $i=1,\dots
  n-1$ we have that $a_i\cdot a_{i+1}= 0$.
\end{lemma}

\begin{proof}
  This is a straightforward artifact of the manner in which graphs are
  labeled: if $I_i$ and $I_{i+1}$ are two consecutive interval in
  $(\partial D)\setminus (\partial\Gamma)$, then the product of last factor in
  $a(I_{i})$ with the first factor in $a(I_{i+1})$ vanishes.
\end{proof}

\begin{remark}
  The actions $\mu^0_n$ correspond to counting tiling graphs that are trees.
\end{remark}

\begin{lemma}
  \label{lem:ActionsAreEven}
  For each operation tree,
  \[ n=2md-4d+2-2w.\]
  In particular, $\mu^w_n=0$ if $n$ is odd.
\end{lemma}

\begin{proof}
  The equation follows by considering the underlying graph $\Gamma$
  for each (centered) operation, and noting that its number of
  vertices is $n+d$, its number of edges is $(m-1)d+n/2$, and its
  Euler characteristic is $1-w$.
\end{proof}

Sometimes, it will be useful to think of the left- and right-extended
operations as counts of extended graphs (as
in~\cite{TorusAlg}). 

\begin{defn}
\label{def:ExtendedGraph} A {\em right-extended graph} is a
graph obtained by taking a centered tiling graph, and inserting a
sequence of ($2$-valent) vertices on the edge $e$ connecting to the
root vertex.  Each $2$-valent vertex has two sectors: the one to the
right of $e$ (thought of as oriented from the boundary of the disk to
an internal vertex), which is unlabeled; and the one to the left
of $e$, which is labeled with a basic algebra element. The algebra
elements labels are uniquely determined: they are taken so that the
product of the algebra elements along the edge (as visible from the
boundary; c.f. Definition~\ref{def:AlgebraSequence}) are non-zero.

A {\em left extended graph} is defined analogously, with the
understanding that now the sector to the left each $2$-valent vertex
is unlabeled, while the one on the right is labeled by an
algebra element. 
\end{defn}

\begin{figure}[ht]
  \input{ExtendedTilings.pstex_t}
  \caption{\label{fig:ExtendedTilings} {\bf{Extended graphs.}}}
\end{figure}
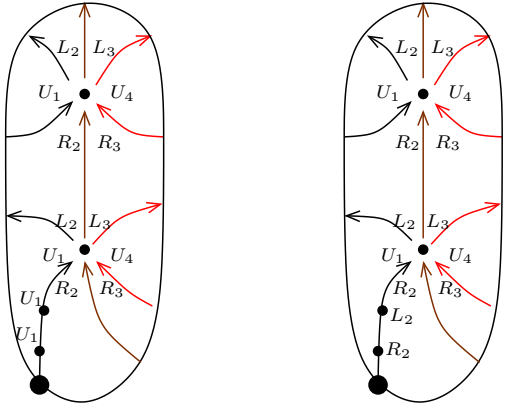

Examples of extended graphs are illustrated in Figure~\ref{fig:ExtendedTilings}.
The figure on the left is right extended, representing the operation
\[ \mu_{10}(R_2,R_3,U_4,U_3,U_4,L_3,L_2,U_1,U_2,U_1^3)=U_1^2;\]
while the one on the right is left-extended, representing
\[ \mu_{10}(U_2 R_2,R_3,U_4,U_3,U_4,L_3,L_2,U_1,U_2,U_1)=U_2\Idemp{1}.\]

%% file: Node.pstex_t
\begin{picture}(0,0)%
\includegraphics{Node.pstex}%
\end{picture}%
\setlength{\unitlength}{1243sp}%
\begingroup\makeatletter\ifx\SetFigFont\undefined%
\gdef\SetFigFont#1#2#3#4#5{%
  \reset@font\fontsize{#1}{#2pt}%
  \fontfamily{#3}\fontseries{#4}\fontshape{#5}%
  \selectfont}%
\fi\endgroup%
\begin{picture}(3198,4893)(1336,-6385)
\put(3601,-4021){\makebox(0,0)[lb]{\smash{{\SetFigFont{7}{8.4}{\rmdefault}{\mddefault}{\updefault}{\color[rgb]{0,0,0}$U_4$}%
}}}}
\put(3241,-3121){\makebox(0,0)[lb]{\smash{{\SetFigFont{7}{8.4}{\rmdefault}{\mddefault}{\updefault}{\color[rgb]{0,0,0}$L_3$}%
}}}}
\put(3331,-5011){\makebox(0,0)[lb]{\smash{{\SetFigFont{7}{8.4}{\rmdefault}{\mddefault}{\updefault}{\color[rgb]{0,0,0}$R_3$}%
}}}}
\put(2521,-3121){\makebox(0,0)[lb]{\smash{{\SetFigFont{7}{8.4}{\rmdefault}{\mddefault}{\updefault}{\color[rgb]{0,0,0}$L_2$}%
}}}}
\put(2521,-5011){\makebox(0,0)[lb]{\smash{{\SetFigFont{7}{8.4}{\rmdefault}{\mddefault}{\updefault}{\color[rgb]{0,0,0}$R_2$}%
}}}}
\put(2161,-4021){\makebox(0,0)[lb]{\smash{{\SetFigFont{7}{8.4}{\rmdefault}{\mddefault}{\updefault}{\color[rgb]{0,0,0}$U_1$}%
}}}}
\put(1351,-2221){\makebox(0,0)[lb]{\smash{{\SetFigFont{7}{8.4}{\rmdefault}{\mddefault}{\updefault}{\color[rgb]{0,0,0}$1$}%
}}}}
\put(1441,-5821){\makebox(0,0)[lb]{\smash{{\SetFigFont{7}{8.4}{\rmdefault}{\mddefault}{\updefault}{\color[rgb]{0,0,0}$1$}%
}}}}
\put(2881,-1771){\makebox(0,0)[lb]{\smash{{\SetFigFont{7}{8.4}{\rmdefault}{\mddefault}{\updefault}{\color[rgb]{.5,.17,0}$2$}%
}}}}
\put(2881,-6271){\makebox(0,0)[lb]{\smash{{\SetFigFont{7}{8.4}{\rmdefault}{\mddefault}{\updefault}{\color[rgb]{.5,.17,0}$2$}%
}}}}
\put(4321,-2221){\makebox(0,0)[lb]{\smash{{\SetFigFont{7}{8.4}{\rmdefault}{\mddefault}{\updefault}{\color[rgb]{1,0,0}$3$}%
}}}}
\put(4321,-5821){\makebox(0,0)[lb]{\smash{{\SetFigFont{7}{8.4}{\rmdefault}{\mddefault}{\updefault}{\color[rgb]{1,0,0}$3$}%
}}}}
\end{picture}%

%% file: Unlabelable.pstex_t
\begin{picture}(0,0)%
\includegraphics{Unlabelable.pstex}%
\end{picture}%
\setlength{\unitlength}{1243sp}%
\begingroup\makeatletter\ifx\SetFigFont\undefined%
\gdef\SetFigFont#1#2#3#4#5{%
  \reset@font\fontsize{#1}{#2pt}%
  \fontfamily{#3}\fontseries{#4}\fontshape{#5}%
  \selectfont}%
\fi\endgroup%
\begin{picture}(3035,4792)(4302,-6660)
\end{picture}%

%% file: TilingExamples.pstex_t
\begin{picture}(0,0)%
\includegraphics{TilingExamples.pstex}%
\end{picture}%
\setlength{\unitlength}{1243sp}%
\begingroup\makeatletter\ifx\SetFigFont\undefined%
\gdef\SetFigFont#1#2#3#4#5{%
  \reset@font\fontsize{#1}{#2pt}%
  \fontfamily{#3}\fontseries{#4}\fontshape{#5}%
  \selectfont}%
\fi\endgroup%
\begin{picture}(19052,8515)(1543,-5944)
\put(3601,-4021){\makebox(0,0)[lb]{\smash{{\SetFigFont{7}{8.4}{\rmdefault}{\mddefault}{\updefault}{\color[rgb]{0,0,0}$U_4$}%
}}}}
\put(3241,-3121){\makebox(0,0)[lb]{\smash{{\SetFigFont{7}{8.4}{\rmdefault}{\mddefault}{\updefault}{\color[rgb]{0,0,0}$L_3$}%
}}}}
\put(3331,-5011){\makebox(0,0)[lb]{\smash{{\SetFigFont{7}{8.4}{\rmdefault}{\mddefault}{\updefault}{\color[rgb]{0,0,0}$R_3$}%
}}}}
\put(2521,-3121){\makebox(0,0)[lb]{\smash{{\SetFigFont{7}{8.4}{\rmdefault}{\mddefault}{\updefault}{\color[rgb]{0,0,0}$L_2$}%
}}}}
\put(2521,-5011){\makebox(0,0)[lb]{\smash{{\SetFigFont{7}{8.4}{\rmdefault}{\mddefault}{\updefault}{\color[rgb]{0,0,0}$R_2$}%
}}}}
\put(2161,-4021){\makebox(0,0)[lb]{\smash{{\SetFigFont{7}{8.4}{\rmdefault}{\mddefault}{\updefault}{\color[rgb]{0,0,0}$U_1$}%
}}}}
\put(3601,479){\makebox(0,0)[lb]{\smash{{\SetFigFont{7}{8.4}{\rmdefault}{\mddefault}{\updefault}{\color[rgb]{0,0,0}$U_4$}%
}}}}
\put(3241,1379){\makebox(0,0)[lb]{\smash{{\SetFigFont{7}{8.4}{\rmdefault}{\mddefault}{\updefault}{\color[rgb]{0,0,0}$L_3$}%
}}}}
\put(3331,-511){\makebox(0,0)[lb]{\smash{{\SetFigFont{7}{8.4}{\rmdefault}{\mddefault}{\updefault}{\color[rgb]{0,0,0}$R_3$}%
}}}}
\put(2521,1379){\makebox(0,0)[lb]{\smash{{\SetFigFont{7}{8.4}{\rmdefault}{\mddefault}{\updefault}{\color[rgb]{0,0,0}$L_2$}%
}}}}
\put(2521,-511){\makebox(0,0)[lb]{\smash{{\SetFigFont{7}{8.4}{\rmdefault}{\mddefault}{\updefault}{\color[rgb]{0,0,0}$R_2$}%
}}}}
\put(2161,479){\makebox(0,0)[lb]{\smash{{\SetFigFont{7}{8.4}{\rmdefault}{\mddefault}{\updefault}{\color[rgb]{0,0,0}$U_1$}%
}}}}
\put(7651,-4021){\makebox(0,0)[lb]{\smash{{\SetFigFont{7}{8.4}{\rmdefault}{\mddefault}{\updefault}{\color[rgb]{0,0,0}$U_4$}%
}}}}
\put(7291,-3121){\makebox(0,0)[lb]{\smash{{\SetFigFont{7}{8.4}{\rmdefault}{\mddefault}{\updefault}{\color[rgb]{0,0,0}$L_3$}%
}}}}
\put(7381,-5011){\makebox(0,0)[lb]{\smash{{\SetFigFont{7}{8.4}{\rmdefault}{\mddefault}{\updefault}{\color[rgb]{0,0,0}$R_3$}%
}}}}
\put(6571,-3121){\makebox(0,0)[lb]{\smash{{\SetFigFont{7}{8.4}{\rmdefault}{\mddefault}{\updefault}{\color[rgb]{0,0,0}$L_2$}%
}}}}
\put(6571,-5011){\makebox(0,0)[lb]{\smash{{\SetFigFont{7}{8.4}{\rmdefault}{\mddefault}{\updefault}{\color[rgb]{0,0,0}$R_2$}%
}}}}
\put(6211,-4021){\makebox(0,0)[lb]{\smash{{\SetFigFont{7}{8.4}{\rmdefault}{\mddefault}{\updefault}{\color[rgb]{0,0,0}$U_1$}%
}}}}
\put(7651,479){\makebox(0,0)[lb]{\smash{{\SetFigFont{7}{8.4}{\rmdefault}{\mddefault}{\updefault}{\color[rgb]{0,0,0}$U_4$}%
}}}}
\put(7291,1379){\makebox(0,0)[lb]{\smash{{\SetFigFont{7}{8.4}{\rmdefault}{\mddefault}{\updefault}{\color[rgb]{0,0,0}$L_3$}%
}}}}
\put(7381,-511){\makebox(0,0)[lb]{\smash{{\SetFigFont{7}{8.4}{\rmdefault}{\mddefault}{\updefault}{\color[rgb]{0,0,0}$R_3$}%
}}}}
\put(6571,1379){\makebox(0,0)[lb]{\smash{{\SetFigFont{7}{8.4}{\rmdefault}{\mddefault}{\updefault}{\color[rgb]{0,0,0}$L_2$}%
}}}}
\put(6571,-511){\makebox(0,0)[lb]{\smash{{\SetFigFont{7}{8.4}{\rmdefault}{\mddefault}{\updefault}{\color[rgb]{0,0,0}$R_2$}%
}}}}
\put(6211,479){\makebox(0,0)[lb]{\smash{{\SetFigFont{7}{8.4}{\rmdefault}{\mddefault}{\updefault}{\color[rgb]{0,0,0}$U_1$}%
}}}}
\put(11656,-2266){\makebox(0,0)[lb]{\smash{{\SetFigFont{7}{8.4}{\rmdefault}{\mddefault}{\updefault}{\color[rgb]{0,0,0}$U_4$}%
}}}}
\put(11296,-1366){\makebox(0,0)[lb]{\smash{{\SetFigFont{7}{8.4}{\rmdefault}{\mddefault}{\updefault}{\color[rgb]{0,0,0}$L_3$}%
}}}}
\put(11386,-3256){\makebox(0,0)[lb]{\smash{{\SetFigFont{7}{8.4}{\rmdefault}{\mddefault}{\updefault}{\color[rgb]{0,0,0}$R_3$}%
}}}}
\put(10576,-1366){\makebox(0,0)[lb]{\smash{{\SetFigFont{7}{8.4}{\rmdefault}{\mddefault}{\updefault}{\color[rgb]{0,0,0}$L_2$}%
}}}}
\put(10576,-3256){\makebox(0,0)[lb]{\smash{{\SetFigFont{7}{8.4}{\rmdefault}{\mddefault}{\updefault}{\color[rgb]{0,0,0}$R_2$}%
}}}}
\put(10216,-2266){\makebox(0,0)[lb]{\smash{{\SetFigFont{7}{8.4}{\rmdefault}{\mddefault}{\updefault}{\color[rgb]{0,0,0}$U_1$}%
}}}}
\put(17551,-4021){\makebox(0,0)[lb]{\smash{{\SetFigFont{7}{8.4}{\rmdefault}{\mddefault}{\updefault}{\color[rgb]{0,0,0}$U_4$}%
}}}}
\put(17281,-5011){\makebox(0,0)[lb]{\smash{{\SetFigFont{7}{8.4}{\rmdefault}{\mddefault}{\updefault}{\color[rgb]{0,0,0}$R_3$}%
}}}}
\put(16471,-3121){\makebox(0,0)[lb]{\smash{{\SetFigFont{7}{8.4}{\rmdefault}{\mddefault}{\updefault}{\color[rgb]{0,0,0}$L_2$}%
}}}}
\put(16471,-5011){\makebox(0,0)[lb]{\smash{{\SetFigFont{7}{8.4}{\rmdefault}{\mddefault}{\updefault}{\color[rgb]{0,0,0}$R_2$}%
}}}}
\put(16111,-4021){\makebox(0,0)[lb]{\smash{{\SetFigFont{7}{8.4}{\rmdefault}{\mddefault}{\updefault}{\color[rgb]{0,0,0}$U_1$}%
}}}}
\put(17551,479){\makebox(0,0)[lb]{\smash{{\SetFigFont{7}{8.4}{\rmdefault}{\mddefault}{\updefault}{\color[rgb]{0,0,0}$U_4$}%
}}}}
\put(17191,1379){\makebox(0,0)[lb]{\smash{{\SetFigFont{7}{8.4}{\rmdefault}{\mddefault}{\updefault}{\color[rgb]{0,0,0}$L_3$}%
}}}}
\put(17281,-511){\makebox(0,0)[lb]{\smash{{\SetFigFont{7}{8.4}{\rmdefault}{\mddefault}{\updefault}{\color[rgb]{0,0,0}$R_3$}%
}}}}
\put(16471,1379){\makebox(0,0)[lb]{\smash{{\SetFigFont{7}{8.4}{\rmdefault}{\mddefault}{\updefault}{\color[rgb]{0,0,0}$L_2$}%
}}}}
\put(16471,-511){\makebox(0,0)[lb]{\smash{{\SetFigFont{7}{8.4}{\rmdefault}{\mddefault}{\updefault}{\color[rgb]{0,0,0}$R_2$}%
}}}}
\put(16111,479){\makebox(0,0)[lb]{\smash{{\SetFigFont{7}{8.4}{\rmdefault}{\mddefault}{\updefault}{\color[rgb]{0,0,0}$U_1$}%
}}}}
\put(17371,-3301){\makebox(0,0)[lb]{\smash{{\SetFigFont{7}{8.4}{\rmdefault}{\mddefault}{\updefault}{\color[rgb]{0,0,0}$L_3$}%
}}}}
\end{picture}%

%% file: WeightedM2.pstex_t
\begin{picture}(0,0)%
\includegraphics{WeightedM2.pstex}%
\end{picture}%
\setlength{\unitlength}{1243sp}%
\begingroup\makeatletter\ifx\SetFigFont\undefined%
\gdef\SetFigFont#1#2#3#4#5{%
  \reset@font\fontsize{#1}{#2pt}%
  \fontfamily{#3}\fontseries{#4}\fontshape{#5}%
  \selectfont}%
\fi\endgroup%
\begin{picture}(4791,8791)(6043,-5286)
\put(7561,1739){\makebox(0,0)[lb]{\smash{{\SetFigFont{7}{8.4}{\rmdefault}{\mddefault}{\updefault}{\color[rgb]{0,0,0}$L_2$}%
}}}}
\put(6661,929){\makebox(0,0)[lb]{\smash{{\SetFigFont{7}{8.4}{\rmdefault}{\mddefault}{\updefault}{\color[rgb]{0,0,0}$U_1$}%
}}}}
\put(6751,-2311){\makebox(0,0)[lb]{\smash{{\SetFigFont{7}{8.4}{\rmdefault}{\mddefault}{\updefault}{\color[rgb]{0,0,0}$U_1$}%
}}}}
\put(7651,-1636){\makebox(0,0)[lb]{\smash{{\SetFigFont{7}{8.4}{\rmdefault}{\mddefault}{\updefault}{\color[rgb]{0,0,0}$L_2$}%
}}}}
\put(7426,-2986){\makebox(0,0)[lb]{\smash{{\SetFigFont{7}{8.4}{\rmdefault}{\mddefault}{\updefault}{\color[rgb]{0,0,0}$R_2$}%
}}}}
\put(8101,-2311){\makebox(0,0)[lb]{\smash{{\SetFigFont{7}{8.4}{\rmdefault}{\mddefault}{\updefault}{\color[rgb]{0,0,0}$U_3$}%
}}}}
\put(7651,-61){\makebox(0,0)[lb]{\smash{{\SetFigFont{7}{8.4}{\rmdefault}{\mddefault}{\updefault}{\color[rgb]{0,0,0}$R_2$}%
}}}}
\put(8101,839){\makebox(0,0)[lb]{\smash{{\SetFigFont{7}{8.4}{\rmdefault}{\mddefault}{\updefault}{\color[rgb]{0,0,0}$U_3$}%
}}}}
\end{picture}%

%% file: ExtendedTilings.pstex_t
\begin{picture}(0,0)%
\includegraphics{ExtendedTilings.pstex}%
\end{picture}%
\setlength{\unitlength}{1243sp}%
\begingroup\makeatletter\ifx\SetFigFont\undefined%
\gdef\SetFigFont#1#2#3#4#5{%
  \reset@font\fontsize{#1}{#2pt}%
  \fontfamily{#3}\fontseries{#4}\fontshape{#5}%
  \selectfont}%
\fi\endgroup%
\begin{picture}(9966,8056)(6043,-5159)
\put(6211,-3886){\makebox(0,0)[lb]{\smash{{\SetFigFont{7}{8.4}{\rmdefault}{\mddefault}{\updefault}{\color[rgb]{0,0,0}$U_1$}%
}}}}
\put(14851,929){\makebox(0,0)[lb]{\smash{{\SetFigFont{7}{8.4}{\rmdefault}{\mddefault}{\updefault}{\color[rgb]{0,0,0}$U_4$}%
}}}}
\put(14491,1829){\makebox(0,0)[lb]{\smash{{\SetFigFont{7}{8.4}{\rmdefault}{\mddefault}{\updefault}{\color[rgb]{0,0,0}$L_3$}%
}}}}
\put(14581,-61){\makebox(0,0)[lb]{\smash{{\SetFigFont{7}{8.4}{\rmdefault}{\mddefault}{\updefault}{\color[rgb]{0,0,0}$R_3$}%
}}}}
\put(13771,1829){\makebox(0,0)[lb]{\smash{{\SetFigFont{7}{8.4}{\rmdefault}{\mddefault}{\updefault}{\color[rgb]{0,0,0}$L_2$}%
}}}}
\put(13771,-61){\makebox(0,0)[lb]{\smash{{\SetFigFont{7}{8.4}{\rmdefault}{\mddefault}{\updefault}{\color[rgb]{0,0,0}$R_2$}%
}}}}
\put(13411,929){\makebox(0,0)[lb]{\smash{{\SetFigFont{7}{8.4}{\rmdefault}{\mddefault}{\updefault}{\color[rgb]{0,0,0}$U_1$}%
}}}}
\put(14401,-1636){\makebox(0,0)[lb]{\smash{{\SetFigFont{7}{8.4}{\rmdefault}{\mddefault}{\updefault}{\color[rgb]{0,0,0}$L_3$}%
}}}}
\put(13726,-1636){\makebox(0,0)[lb]{\smash{{\SetFigFont{7}{8.4}{\rmdefault}{\mddefault}{\updefault}{\color[rgb]{0,0,0}$L_2$}%
}}}}
\put(14851,-2311){\makebox(0,0)[lb]{\smash{{\SetFigFont{7}{8.4}{\rmdefault}{\mddefault}{\updefault}{\color[rgb]{0,0,0}$U_4$}%
}}}}
\put(14626,-2986){\makebox(0,0)[lb]{\smash{{\SetFigFont{7}{8.4}{\rmdefault}{\mddefault}{\updefault}{\color[rgb]{0,0,0}$R_3$}%
}}}}
\put(13501,-2311){\makebox(0,0)[lb]{\smash{{\SetFigFont{7}{8.4}{\rmdefault}{\mddefault}{\updefault}{\color[rgb]{0,0,0}$U_1$}%
}}}}
\put(13726,-2986){\makebox(0,0)[lb]{\smash{{\SetFigFont{7}{8.4}{\rmdefault}{\mddefault}{\updefault}{\color[rgb]{0,0,0}$R_2$}%
}}}}
\put(13681,-3526){\makebox(0,0)[lb]{\smash{{\SetFigFont{7}{8.4}{\rmdefault}{\mddefault}{\updefault}{\color[rgb]{0,0,0}$L_2$}%
}}}}
\put(13591,-4201){\makebox(0,0)[lb]{\smash{{\SetFigFont{7}{8.4}{\rmdefault}{\mddefault}{\updefault}{\color[rgb]{0,0,0}$R_2$}%
}}}}
\put(8101,929){\makebox(0,0)[lb]{\smash{{\SetFigFont{7}{8.4}{\rmdefault}{\mddefault}{\updefault}{\color[rgb]{0,0,0}$U_4$}%
}}}}
\put(7741,1829){\makebox(0,0)[lb]{\smash{{\SetFigFont{7}{8.4}{\rmdefault}{\mddefault}{\updefault}{\color[rgb]{0,0,0}$L_3$}%
}}}}
\put(7831,-61){\makebox(0,0)[lb]{\smash{{\SetFigFont{7}{8.4}{\rmdefault}{\mddefault}{\updefault}{\color[rgb]{0,0,0}$R_3$}%
}}}}
\put(7021,1829){\makebox(0,0)[lb]{\smash{{\SetFigFont{7}{8.4}{\rmdefault}{\mddefault}{\updefault}{\color[rgb]{0,0,0}$L_2$}%
}}}}
\put(7021,-61){\makebox(0,0)[lb]{\smash{{\SetFigFont{7}{8.4}{\rmdefault}{\mddefault}{\updefault}{\color[rgb]{0,0,0}$R_2$}%
}}}}
\put(6661,929){\makebox(0,0)[lb]{\smash{{\SetFigFont{7}{8.4}{\rmdefault}{\mddefault}{\updefault}{\color[rgb]{0,0,0}$U_1$}%
}}}}
\put(7651,-1636){\makebox(0,0)[lb]{\smash{{\SetFigFont{7}{8.4}{\rmdefault}{\mddefault}{\updefault}{\color[rgb]{0,0,0}$L_3$}%
}}}}
\put(6976,-1636){\makebox(0,0)[lb]{\smash{{\SetFigFont{7}{8.4}{\rmdefault}{\mddefault}{\updefault}{\color[rgb]{0,0,0}$L_2$}%
}}}}
\put(8101,-2311){\makebox(0,0)[lb]{\smash{{\SetFigFont{7}{8.4}{\rmdefault}{\mddefault}{\updefault}{\color[rgb]{0,0,0}$U_4$}%
}}}}
\put(7876,-2986){\makebox(0,0)[lb]{\smash{{\SetFigFont{7}{8.4}{\rmdefault}{\mddefault}{\updefault}{\color[rgb]{0,0,0}$R_3$}%
}}}}
\put(6751,-2311){\makebox(0,0)[lb]{\smash{{\SetFigFont{7}{8.4}{\rmdefault}{\mddefault}{\updefault}{\color[rgb]{0,0,0}$U_1$}%
}}}}
\put(6976,-2986){\makebox(0,0)[lb]{\smash{{\SetFigFont{7}{8.4}{\rmdefault}{\mddefault}{\updefault}{\color[rgb]{0,0,0}$R_2$}%
}}}}
\put(6301,-3166){\makebox(0,0)[lb]{\smash{{\SetFigFont{7}{8.4}{\rmdefault}{\mddefault}{\updefault}{\color[rgb]{0,0,0}$U_1$}%
}}}}
\end{picture}%

%% file: geomrel.tex
\section{Verifying the $A_\infty$ relations}
\label{sec:Relation}
Endow $\Clg[t]$ with operations
\[ \mu^{\vec{w}}_n\colon \Clg[t]^{\otimes n}\to \Clg[t] \] as in
Equation~\ref{eq:WeightedAinftyAlgebra}. 

Our aim in this section is to prove the following:

\begin{thm}
  \label{thm:WeightedAinftyAlgebra}
  The operations $\{\mu^{\vec{w}}_n\}$ give $\Clg[t]$ the structure of a unital, weighted $\Ainfty$ algebra.
\end{thm}

Theorem~\ref{thm:WeightedAinftyAlgebra} states that sums of
double-composite trees vanish.  We give a combinatorial proof in the
spirit of~\cite[Section~3.3]{TorusAlg}.  The key point is to give a
fixed point free involution on the set of terms that count terms in
the $\Ainfty$ relation (preserving the inputs and the outputs),
thereby specifying pairs of terms that cancel. 

Continuing as before, we are working in characteristic $2$. See Theorem~\ref{thm:WeightedAinftyAlgebraZ} for the corresponding result over $\Z$.

\subsection{The unweighted case}

We separate out a special case where $\vec{w}=0$; equivalently, where
we restrict to counting those (centered) tiling graphs that are trees.
Although this is a special case of Theorem~\ref{thm:WeightedAinftyAlgebra},
we state it as a separate result, which we prove in this subsection:

\begin{thm}
  \label{thm:unWeightedAinftyAlgebra}
  The operations $\{\mu^{0}_n\}$ give $\Clg[t]$ the structure of
  a unital $\Ainfty$ algebra. 
\end{thm}

\begin{cor}
  The operations $\mu^0$ endow $\Clg[t]$ with the non-trivial
  $\Ainfty$ structure specified in Theorem~\ref{thm:UniqueCm}.
\end{cor}

Assuming Theorem~\ref{thm:unWeightedAinftyAlgebra}, the above corollary
follows quickly, as follows:

\begin{proof}    
  Theorem~\ref{thm:unWeightedAinftyAlgebra} ensures that
  $\mu^0$ is an $\Ainfty$ algebra. 
  Lemma~\ref{lem:GradedOperations}
  ensure that the grading
  conditions required by Theorem~\ref{thm:UniqueCm} hold.
  It is straightforward to compute
  \[ \mu_{2m-2}(U_1,R_2,\dots,R_{m-1},U_m,L_{m-1}\dots,L_2)=\Idemp{1}t.\]
  (This is a centered graph with a single vertex.)  Thus,
  Theorem~\ref{thm:UniqueCm} ensures the stated corollary.
\end{proof}

The proof of Theorem~\ref{thm:unWeightedAinftyAlgebra}
will occupy the rest of the present subsection.

\begin{defn}\label{def:CenteredCompositeTypes}
Non-trivial composite trees with inputs in $\Clg_+$
and output in $\IdempRing[t]$ can be of the following {\em centered composite types}:
\begin{itemize}
\item
A node labeled $\mu_2$ on the top and a tree contributing to
  $\mu_n$ on the bottom.  We call these terms of type $(C2)$.
\item
  A node with $\mu_n$ on the bottom and a left extended $\mu_{\ell}$
  on the top. We call these terms of type $[CL]$. Further subdivide
  these into $[CL+]$, when the action on the top level feeds into
  the leftmost input on the bottom; $[CL-]$ when it feeds into the
  rightmost input on the bottom; and $[CLg]$ when it feeds into any
  other term.
\item
 A node with (centered) $\mu_n$ on the bottom and a right extended
  $\mu_{\ell}$ on the top. We call these $[CR]$. Further
  subdivide these into $[CR+]$, $[CR-]$, and $[CRg]$ according to
  whether the top action feeds into the the leftmost, rightmost, or
  other term on the bottom. 
\end{itemize}
\end{defn}

(We are following here the notational conventions
from~\cite[Section~3.3]{TorusAlg}.)
Terms of type $[CL]$ or $[CR]$ are referred to simply as terms of type $[C*]$.

Graphically, terms of Type $(C2)$ are represented by a centered tiling
graph, together with a (dotted) path connecting $(\partial D)\setminus
(\partial\Gamma)$, that separates the two sectors where the algebra
elements are multiplied. For example, the pictures on the left of Figure~\ref{fig:PullingOutEx} represent the terms of type $(C2)$ contributing to
\[ \mu_8(U_1,U_2,U_1^2,R_2,U_3,U_2,\mu_2(U_3,U_3),L_2)=t^3 \Idemp{1} \]
and
\[ \mu_8(U_1,\mu_2(R_2,L_2),U_1^2,R_2,U_3,U_2,U_3^2,L_2)=t^3\Idemp{1}\] 
respectively.

\begin{figure}[ht]
\input{PullingOutEx.pstex_t}
\caption{\label{fig:PullingOutEx} {\bf{Terms of type $(C2)$ and $[C*]$}}}
\end{figure}
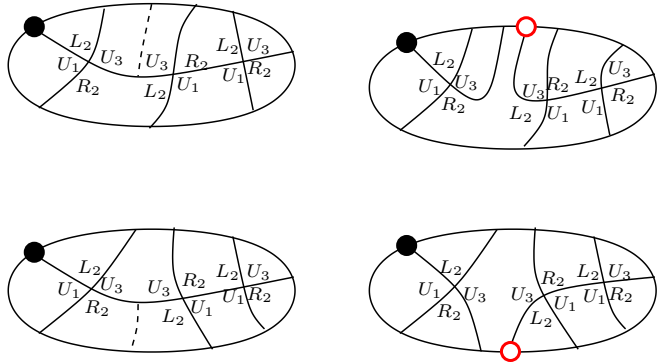

Terms of type $[C*]$ correspond to counts of (cycle-free) rooted, planar,
two-component graphs whose two components are ordered $\Gamma_1$ and
$\Gamma_2$ (``bottom'' and ``top'' respectively) -- with the
additional hypothesis that the root of $\Gamma_2$ is adjacent along
the boundary to a leaf or root of $\Gamma_1$, subject to the following
labeling constraint. Choose an arc $S$ that connects the root of
$\Gamma_2$ with the nearest point in $\Gamma_1$.  The algebra element
marking the first sector before $S$ (with respect to the counterclockwise
ordering) must multiply with the the algebra element marking the first
sector after $S$.

To read off the sequence of algebra elements, we use the following
convention. Draw an arc in $\partial D$ from the root of $\Gamma_2$
to the nearest point in $\Gamma_1$. Denote that arc $S$.  Now, we read
off the sequence of algebra elements as in
Definition~\ref{def:AlgebraSequence}, with the convention that the
intervals now are the components of intervals of $\partial D\setminus
(\Gamma\cap \partial D)\setminus S$. 

\begin{defn}
  \label{def:CompositePattern}
These graphical representations of terms of type $[C*]$ are called
{\em (centered) composite patterns}.
\end{defn}

Examples are given on the right in 
Figure~\ref{fig:PullingOutEx}. In the pictures, the
ordering on $\Gamma_1$ and $\Gamma_2$ is indicated by drawing the root
vertex of $\Gamma_1$ as a solid black dot and the root of $\Gamma_2$ as a hollow
(red) dot.
These illustrate the terms of type $[C*]$
\[ \mu_4(U_1,\mu_6(U_2,U_1^2,R_2,U_3,U_2,U_3),U_3,L_2)=\Idemp{1} t^3 \]
and
\[ 
\mu_4(U_1,R_2,\mu_6(L_2,U_1^2,R_2,U_3,U_2,U_3^2),L_2)=\Idemp{1} t^3 \]
respectively.

Note that the precise type -- $[CL]$ or $[CR]$ -- of the diagram
representing a term of type $[C*]$ can be seen in the planar geometry
of the picture, as follows. Consider the two-component planar graph
$\Gamma_1\cup \Gamma_2$ representing the pair of operations. Travel
the boundary $\partial D$ counterclockwise starting at the root vertex of
$\Gamma_1$. If the first vertex in $\Gamma_2$ encountered is the root
of $\Gamma_2$ (as in the bottom right picture in Figure~\ref{fig:PullingOutEx}),
then the operation containing $\Gamma_2$ is right-extended;
otherwise, the operation containing $\Gamma_2$ is left-extended.

\begin{lemma}
  \label{lem:CancelC2w0}
  With weight $0$,
  terms of type $(C2)$ cancel against terms of type 
  $[C*]\setminus [CL-]\cup[CR+]$.
\end{lemma}

\begin{proof}
  Terms of type $(C2)$ correspond to a centered tiling graph $\Gamma$ and a
  path $P$ that connects $\partial D$ to some internal edge $e$ in
  $\Gamma$.  

  The cancelling term of type $[C*]$ is obtained by pulling the edge
  $e$ out to the boundary along $P$, and letting $S$ denote the newly
  created path in $\partial D$. Since the original graph $P$ has no
  loops (this is the weight zero hypothesis), the newly created graph
  has two components $\Gamma_1\cup \Gamma_2$.  Label $\Gamma_1$ so
  that it contains the root vertex. The root vertex of $\Gamma_2$,
  then is the point in $\Gamma_2$ that meets $S$. See
  Figure~\ref{fig:PullingOut}.  We have specified a map from terms of
  type $(C2)$ to terms of type $[C*]$.  Observe that for the terms
  obtained in this manner, the arc $S$ never connects the two root
  vertices (since the newly created vertex of $\Gamma_1$ is guaranteed
  not to be a root). The terms where $S$ connects the two root vertices
  are precisely those that represent $[CL-]$ and $[CR+]$.

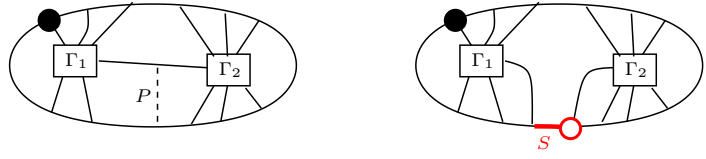
\begin{figure}[ht]
\input{PullingOut.pstex_t}
\caption{\label{fig:PullingOut} {\bf{Cancelling of terms of type
      $(C2)$.}}  Factoring an input to a centered action corresponds
  to an arc (dashed) connecting an internal edge to the
  boundary. Pulling that out to the boundary gives a cancelling term, as shown.}
\end{figure}

The inverse to this ``pulling out'' operation is obtained by pushing
the path $S$ into the interior. This pushing in operation fails to
give a valid rooted tree in the special case where $S$ connects the roots of
$\Gamma_1$ and $\Gamma_2$; i.e. terms of type $[CL-]$ and $[CR+]$.

  The two rows of Figure~\ref{fig:PullingOutEx} illustrate the two relations
  \begin{align*}
    \mu_8(U_1,U_2,U_1^2,R_2,U_3,U_2,\mu_2(U_3,U_3),L_2) 
    &=
    \mu_4(U_1,\mu_6(U_2,U_1^2,R_2,U_3,U_2,U_3),U_3,L_2) \\
    \mu_8(U_1,\mu_2(R_2,L_2),U_1^2,R_2,U_3,U_2,U_3^2,L_2)
    &=
    \mu_4(U_1,R_2,\mu_6(L_2,U_1^2,R_2,U_3,U_2,U_3^2),L_2)
  \end{align*}
\end{proof}

Lemma~\ref{lem:CancelC2w0} leaves terms of type $[CL-]\cup[CR+]$ unaccounted for. These terms are paired off in the following:

\begin{lemma}
  \label{lem:CLpCLmw0}
  With weight $0$,
  terms of type $[CL-]$ cancel against terms of type $[CR+]$.
\end{lemma}

\begin{proof}
  As in the proof of Lemma~\ref{lem:CancelC2w0}, we represent these
  terms as planar graphs $\Gamma$ with two components $\Gamma_1$ and
  $\Gamma_2$, so that the roots of $\Gamma_1$ and $\Gamma_2$ are
  adjacent in $\partial D$. The involution is now realized by
  switching ordering of the two roots.  See Figure~\ref{fig:CLpCRm}
  for a schematic illustration. See also Figure~\ref{fig:CLpCRmEx} for
  an explicit example, illustrating the relation:
  \[\mu_6(L_2,U_1^2,R_2,U_3,U_2,\mu_4(U_3^2,L_2,U_1,R_2))=
  \mu_4(\mu_6(L_2,U_1^2,R_2,U_3,U_2,U_3^2),L_2,U_1,R_2)\]

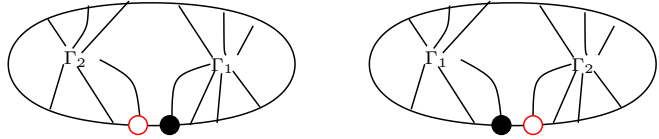
\begin{figure}[ht]
\input{CLpCRm.pstex_t}
\caption{\label{fig:CLpCRm} {\bf{Cancelling of terms of type
      $[CL-]$ and $[CR+]$.}}}
\end{figure}
\begin{figure}[ht]
\input{CLpCRmEx.pstex_t}
\caption{\label{fig:CLpCRmEx} {\bf{Cancelling of terms of type
      $[CL-]$ and $[CR+]$: example.}}}
\end{figure}
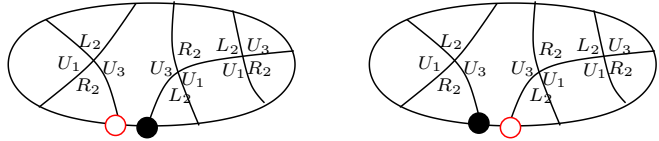
\end{proof}

It is not difficult to see that we have verified the $\Ainfty$
relation for terms that land in $\IdempRing[t]\subset \Clg[t]$.

\begin{defn}
  \label{def:ExtendedComposite}
  Non-trivial composite trees with inputs in $\Clg_+$ and output in
  $\Clg[t]\setminus\IdempRing[t]$ can be classified into the following kinds of
  {\em extended composite types}:
\begin{itemize}
\item A node labeled $\mu_2$ on the top and a left- or right-extended tree on the bottom. These terms are labeled $(L2)$ and $(R2)$ respectively.
  As usual, we subdivide these further according to where the $\mu_2$ is fed into;
  e.g. for $(L2+)$, the $\mu_2$ is fed into the first input to the lower node.
\item A node with a $\mu_2$ on the bottom and $\mu_n$ on the top.
  These are labeled $(2*+)$, when the $\mu_n$ is fed into the first term in $\mu_2$, and $(2*-)$, when the $\mu_n$ is fed into the second. Here,
  $*\in \{L,C,R\}$, according to whether it is left-extended,
  centered, or right-extended.
\item A node with a left- or right-extended $\mu_n$ on the bottom and
  a $\mu_\ell$ with $\ell>2$ on the top. These are labeled $[LL*]$,
  $[LR*]$, $[RL*]$, or $[RR*]$ (the first resp. second letter
  indicates whether the bottom resp. top node is left- or
  right-extended. The element $*$ can be $-$, $g$, or $+$. The label
  $*=-$ indicates that the top node is channelled into the rightmost
  input of the bottom node; the label $*=g$ indicates that the top
  node is channelled into the input which is neither
  leftmost or rightmost (i.e. it is ``general''); while $*=+$
  indicates that the top node is channelled into the leftmost input
  node of the bottom node.
\end{itemize}
\end{defn}

It will be helpful to have a graphical representation of these terms.
Terms of type $(2*+ )$ are simply represented by the tree representing
the $\mu_n$ operation (which can now be extended, as in
Definition~\ref{def:ExtendedGraph}) and the algebra input that feeds
into (right in) the $\mu_2$. We draw the algebra element to the right
of the tree.  Terms of type $(2*-)$ are represented analogously; except in that
case, the algebra element is drawn to the left of the tree.

For example,  the picture on the right of the middle row
of Figure~\ref{fig:LLmEx} represents a term of type $(2C+)$ of the form
$\mu_2(\mu_6(U_2,U_1,R_2,U_3^2,L_2,U_1),R_2)$.

To represent terms of the third type -- $[LL*]$, $[LR*]$, $[RL*]$, or
$[RR*]$ -- extend the notion of composite patterns from
Definition~\ref{def:CompositePattern}, in the spirit of 
Definition~\ref{def:ExtendedGraph}:

\begin{defn} 
  A {\em left-} resp. {\em right- extended composite pattern} is a
  rooted, planar, two-component graph $\Gamma_1\cup \Gamma_2$, with a
  distinguished root on $\Gamma_1$; represented in the pictures by a
  black dot.
\end{defn}

For example, the picture on the middle left of Figure~\ref{fig:LLmEx}
represents a term of type $[LL-]$ representing the composite
$\mu_4(U_2,U_1,R_2,\mu_4(U_3^2,L_2,U_1,R_2))$.

\begin{lemma}
  \label{lem:CancelL2w0}
  With weight $0$, terms of type $(L2)$ cancel against terms of type
  $(2L-)\cup [L**]\setminus([LL-]\cup[LR+])$.  Similarly, terms of
  type $(R2)$ cancel against terms of type
  $(2R+)\cup[R**]\setminus([RL-]\cup[RR+])$.
\end{lemma}  

\begin{proof}
  Consider a term of type $(L2)$.  Let $\Gamma$ be tree at the core of
  the $L$ operation. The proof of Lemma~\ref{lem:CancelC2w0} gives a
  cancelling term of type $[L**]\setminus[LL-]\cup[LR+]$, except in
  the case where the term is of type $(L2+)$ and the leftmost input
  $a_1'$ into $\Gamma$ has $|a|<|a_1'|$, where here $|a|$ is the output
  algebra element of the bottom node. In that case, the
  cancelling term is of type $(2L-)$.

  As an example of this latter type of cancellation, the term of type
  $(L2)$ representing the non-trivial term in
  $\mu_4(\mu_2(U_1,U_1),R_2,U_3,L_2)$ (which outputs $U_1$) cancels
  against the term of type $(2L+)$:
  $\mu_2(U_1,\mu_4(U_1,R_2,U_3,L_2))$.
  
  The stated cancellation of terms of type $(R2)$ follow analogously.
\end{proof}  

Lemma~\ref{lem:CLpCLmw0} has the following analogue:

\begin{lemma}
  \label{lem:RLmLRpw0}
  With weight $0$, Terms of type $[RL-]$ cancel against terms of type $[LR+]$.
\end{lemma}

\begin{proof}
  This again switches the order of the two roots, as in the proof of
  Lemma~\ref{lem:CLpCLmw0}.
  An example is the relation
  \[ \mu_4(\mu_4(U_1,R_2,U_3,U_2^2),U_3,L_2,U_1)=
  \mu_4(U_1,R_2,U_3,\mu_4(U_2^2,U_3,L_2,U_1)).\]
\end{proof}

In the following lemma, we consider $\Ainfty$ relation with input
sequence $a_1,\dots,a_n$ with $a_i\cdot a_{i+1}=0$ for all
$i=1,\dots,n-1$ (and weight $0$). Note that terms of types $[LL-]$ and
$[RR+]$ automatically satisfy this condition.  Terms of type $(2C-)$,
$(2C+)$, $(2L+)$, and $(2R-)$ do not. Although we give a fairly coarse
cancellation statement (6 types of terms cancel), the proof actually
gives a more specific cancellation of pairs.

\begin{lemma}
  \label{lem:LLmRRpw0}
  Let $a$ be the output element to an $\Ainfty$ relation 
  with input sequence $a_1,\dots,a_n$ with $a_i\cdot a_{i+1}=0$ for
  all $i=1,\dots,n-1$ and weight $0$. For terms of weight $0$, we have cancellations of terms of the following types:
  \[ [LL-], [RR+], (2C-), (2C+), (2L+), (2R-).\]
\end{lemma}

\begin{proof}
  There are nine cases of cancellation, according to the relative
  lengths of $a_1$, $a_n$, and
  $a$. The cancellations are spelled out in the following table
  (exactly as in~\cite[Table~2]{TorusAlg}):

  \begin{tabular}{r||c|c|c}
      & $|a|<|a_1|$ & $|a|=|a_1|$ & $|a|>|a_1|$ \\ 
  \hline 
  \hline
  $|a|<|a_n|$ & $[LL-]$
  $\Leftrightarrow$ $[RR+]$
            & $(2C-)$ $\Leftrightarrow$ $[RR+]$
              & $(2R-)$ $\Leftrightarrow$ $[RR+]$ 
\\
\hline 
  $|a|=|a_n|$ & $[LL-]$ 
$\Leftrightarrow$ $(2C+)$
            & $(2C-)$ $\Leftrightarrow$ $(2C+)$
              & $(2R-)$ $\Leftrightarrow$ $(2C+)$\\
\hline
  $|a|>|a_n|$ & $[LL-]$ 
  $\Leftrightarrow$ $(2L+)$
            & $(2C-)$ $\Leftrightarrow$ $(2L+)$
& $(2R-)$ $\Leftrightarrow$ $(2L+)$
\end{tabular}

Given a term of type $[LL-]$ (in the case where $|a|<|a_1|$), the
cancelling term is found by pushing in the edge between the roots of
$\Gamma_1$ and $\Gamma_2$, and then either pulling out another edge
between them (when $|a|<|a_n|$); or moving the root to the next available
leaf; see Figure~\ref{fig:LLm}. 
See also Figure~\ref{fig:LLmEx} for examples in the three subcases, illustrating
the cancellations:
\begin{align*}
 \mu_4(U_2,U_1,R_2,\mu_6(U_3^2,L_2,U_1^2,R_2,U_3,U_2))
&= \mu_4(\mu_6(U_2,U_1,R_2,U_3^2,L_2,U_1^2),R_2,U_3,U_2)) \\
 \mu_4(U_2,U_1,R_2,\mu_4(U_3^2,L_2,U_1,R_2))
&= \mu_2(\mu_6(U_2,U_1,R_2,U_3^2,L_2,U_1),R_2)) \\
 \mu_4(L_2 U_2,U_1,R_2,\mu_4(U_3^2,L_2,U_1,R_2))
&= \mu_2(\mu_6(L_2 U_2,U_1,R_2,U_3^2,L_2,U_1),R_2)) \\
\end{align*}

\begin{figure}[ht]
\input{LLm.pstex_t}
\caption{\label{fig:LLm} {\bf{Cancelling terms of type $[LL-]$.}}}
\end{figure}
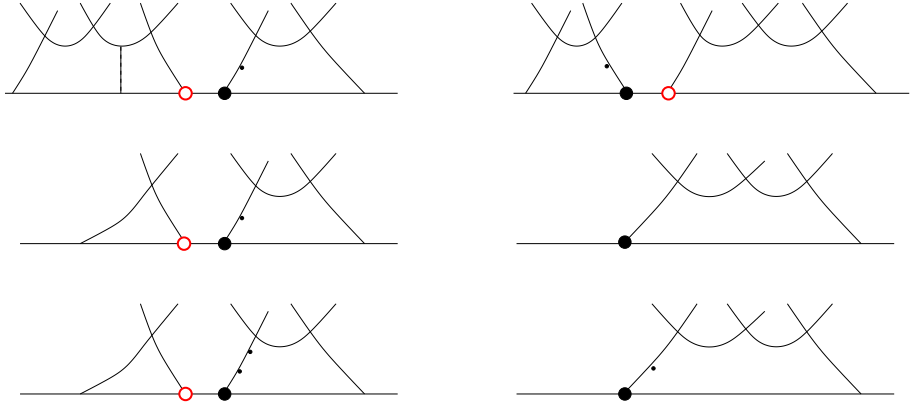

\begin{figure}[ht]
\input{LLmEx.pstex_t}
\caption{\label{fig:LLmEx} {\bf{Examples of cancellation of terms of type $[LL-]$.}} On the top row, the cancelling term is of type $[RR+]$; in the middle it 
is of type $(2C+)$, and on the bottom, it is of type $(2L+)$}
\end{figure}
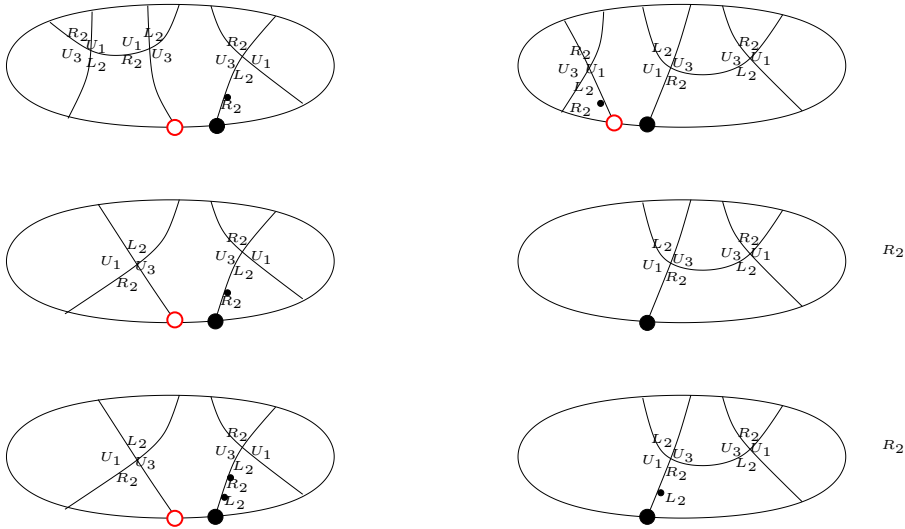

Symmetrically, given a term of type $[RR+]$ the cancelling term is found by 
pushing in the distinguished edge and then pulling out another edge; or moving the root. 

In the four remaining cases (i.e. where $|a|\geq \max(|a_1|,|a_n)$),
the cancellation is obtained by moving the root of the operation. 
See Figure~\ref{fig:MoveRoot} (for an example where $|a|=|a_1|=|a_n|$). 
That figure in turn illustrates the relation
\begin{align*}
\mu_2(U_1,(&\mu_{10}(R_2, R_3, U_4, U_3, U_4, L_3, L_2, U_1, U_2, U_1)) \\
&=\mu_2(\mu_{10}(U_1,R_2, R_3, U_4, U_3, U_4, L_3, L_2, U_1, U_2), U_1)) 
\end{align*}

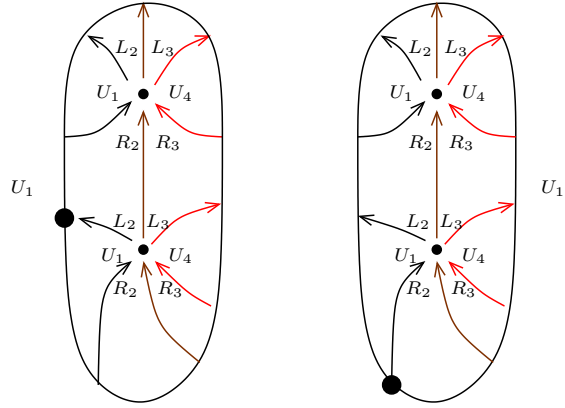
\begin{figure}[ht]
\input{MoveRoot.pstex_t}
\caption{\label{fig:MoveRoot} {\bf{$(2C+)$ and $(2C-)$ cancel.}}}
\end{figure}

\end{proof}

\begin{proof}[of Theorem~\ref{thm:unWeightedAinftyAlgebra}]
  By construction an operation $\mu_n(a_1,\dots,a_n)$ vanishes if 
  $n>2$ and $|a_i|=0$ for any $i$. Unitality follows readily.

  We turn next to the $\Ainfty$ relation.

  Each $\Ainfty$ relation corresponds to a sequence $(a_1,\dots,a_n)$
  of incoming algebra elements.

  The complexity of the $\Ainfty$ relation can be
  measured by the number $n$ of inputs, and then the number $\ell$ of
  consecutive pairs of inputs with non-trivial $\mu_2$. For example,
  there is an $\Ainfty$ relation with three inputs
  $(R_2,L_2,R_2)$. For this relation,  $n=3$ and $\ell=2$.
  
  Clearly, when $n=3$ the $\Ainfty$ relation holds because $\Clg(m,1)$
  is an associative algebra.

  Consider $n>3$. The case where some $a_i$ is an idempotent is special
  since our actions are strictly unital. The only non-trivial case is when
  either $a_1=1$ or $a_n=1$. When $a_1=1$, the two (cancelling) non-trivial terms are
  \[ \mu_2(1,\mu_{n-1}(a_2,\dots,a_n))=\mu_{n-1}(\mu_2(1,a_2),a_3,\dots,a_n).\]
  The case where $a_n=1$ works similarly.

  It now suffices to verify the $A_\infty$ relations with $n>3$, 
  where the sequence $a_1,\dots,a_n$ is of pure elements in $\Clg_+$.

  Suppose that $n> 3$ and $\ell>1$. Then by
  Lemma~\ref{lem:NonMultiplyable}, there are no non-trivial terms in
  the $\Ainfty$ relation.

  Consider next the case where $n>3$ and $\ell=1$.  When the output is
  a power of $t$, Lemma~\ref{lem:CancelC2w0} gives the cancellations
  between terms of type $[C*]\setminus [CL+]\cup[CR-]$ and those of
  type $(C2)$. The remaining cancellations of terms of type $[CL+]$
  against terms of type $[CR-]$ are provided by
  Lemma~\ref{lem:CLpCLmw0}. 

  When the output is not just a power of $t$, Lemma~\ref{lem:CancelL2w0}
  gives the cancellations of terms of type $(L2)$ and $(R2)$ against
  terms of type $[L**]\cup[R**]\setminus ([LL-]\cup[LR+]\cup[RR+]\cup[RL-])$.
  Observe that the terms of excluded type all have $\ell=0$ in the input sequence. Thus, we have completed the proof when $\ell=1$.

  When $n>3$ and $\ell=0$,
  Lemmas~\ref{lem:CLpCLmw0},~\ref{lem:RLmLRpw0},
  and~\ref{lem:LLmRRpw0} give the desired cancellations.
\end{proof}

\subsection{Verifying the $A_\infty$ relation in the weighted case}

We modify the above discussion to take into account possible weights,
starting again from the case where the output of the operation lies in
$\IdempRing[t]$. In the unweighted case, we listed the three
``centered composite types'' in
Definition~\ref{def:CenteredCompositeTypes}.  In the weighted case,
these are to be modified with the understanding that instead of trees,
we are using the more general centered tiling graphs of
Definition~\ref{def:CenteredTilingGraph}. We also add one more type of composite tree:
\begin{itemize}
  \item  A node labeled $\mu^w_0$ on top 
    (which can be any of the operations from Equation~\eqref{eq:AddedMu0s})
    and a tree labeled $\mu_n^{w'}$ on the bottom. Such terms are said to be of type $[C0]$.
    Again, these are subdivided in to types $[C0-]$, $[C0+]$, and $[C0g]$, according to whether
    the $\mu^{w}_0$ term is channeled into the rightmost, 
    leftmost, or other input.
\end{itemize}

Graphically, these terms are represented by a centered tiling graph,
together with a distinguished boundary arc whose associated algebra
element is one of the $U_i$. We find it convenient to label this distinguished
boundary arc by $S$. See for example Figure~\ref{fig:Mu0}; the
term on the left, of type $[C0g]$, represents
\[ \mu_{10}(U_1,R_2,R_3, U_4, \mu_0^{\eVec{3}}, U_4, L_3,L_2, U_1,U_2)=t^2\Idemp{[2]} \]
while the one on the right
\[ \mu_{10}(\mu_0^{\eVec{1}},R_2,R_3, U_4, U_3,U_4,L_3,L_2, U_1,U_2)=t^2\Idemp{[2]} \]
(Throughout this section we suppress weights when they vanish; so, the
$\mu_{10}$ operations appearing above are, in fact,
$\mu_{10}^{0}$.)

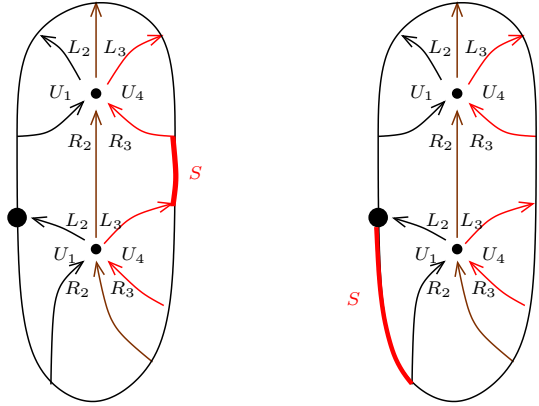
\begin{figure}[ht]
\input{Mu0.pstex_t}
\caption{\label{fig:Mu0} {\bf{Terms of type $[C0g]$ and $[C*]$}}}
\end{figure}

When we write terms of type $[C**]$, they include all of $[C0*]$, $[CL*]$, and $[CR*]$. 

\begin{lemma}
  \label{lem:CancelC2}
  With weight $0$,
  terms of type $(C2)$ cancel against terms of type 
  $[C**]\setminus ([CL+]\cup[CR-]\cup [C0+]\cup[C0-])$.
\end{lemma}

\begin{proof}
  As in the proof of Lemma~\ref{lem:CancelC2w0}, the term that cancels
  a given term of type $(C2)$ is obtained by pushing out an edge. Only
  now, since the original may have loops, the newly-created graph can
  have one or two components. When it has one component, the edge must
  meet a cycle which is in fact a length one cycle by the hypotheses
  of centered tiling graphs
  (Definition~\ref{def:CenteredTilingGraph}). It is easy to see that
  in this case, the newly exposed sector is labeled by a single
  $U_i$.

  The newly exposed boundary arc can now be denoted by $S$. Evidently, $S$ is disjoint from the root
  vertex (i.e. we are hitting terms that are not of type $[C0+]$ or $[C0-]$).

  Conversely, the inverse is once again constructed by pushing $S$ into the interior.
  This creates a centered tiling pattern, provided that the root does not meet $S$. 
  See Figure~\ref{fig:Mu0cancel}.

  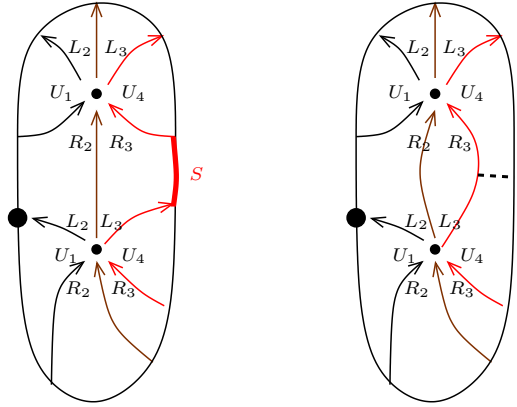
\begin{figure}[ht]
    \input{Mu0cancel.pstex_t}
    \caption{\label{fig:Mu0cancel} {\bf{Pulling out the dashed edge on the right yields the term of type $[C0g]$ on the left.}}}
  \end{figure}
\end{proof}

The poof of Lemma~\ref{lem:CLpCLmw0} adapts readily to the weighted case to give:

\begin{lemma}
  \label{lem:CLpCLm}
  Terms of type $[CL+]$ cancel against terms of type $[CR-]$.
  Similarly, terms of Type $[C0+]$ cancel against those of type $[C0-]$.
\end{lemma}

\begin{proof}
  The cancellation of $[CL+]$ and $[CR-]$ is as in the proof of
  Lemma~\ref{lem:CLpCLmw0}. The terms of type $[C0+]\cup [C0-]$ are
  those where the the distinguished boundary arc $S$ contains the
  root. These terms come in pairs, according to which boundary
  component of $S$ is marked by the root marker; see Figure~\ref{fig:C0pC0m}
  for a picture of the cancellation
  \begin{align*} 
    \mu_{10}(R_2,R_3,U_4,R_3,&U_4,L_3,L_2,U_1,U_2,\mu_0^{\eVec{1}})\\
    &=\mu_{10}(\mu_0^{\eVec{1}},R_2,R_3,U_4,R_3,U_4,L_3,L_2,U_1,U_2).
  \end{align*}
  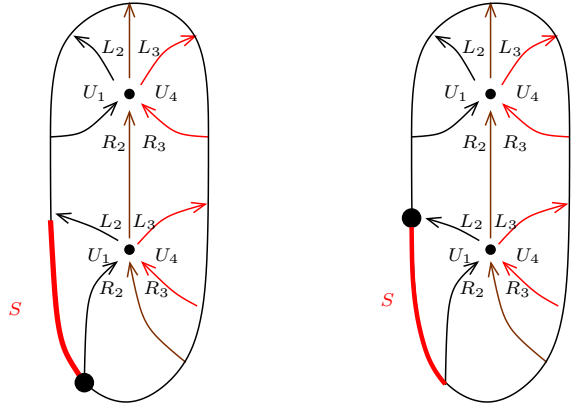
\begin{figure}[ht]
    \input{C0pC0m.pstex_t}
    \caption{\label{fig:C0pC0m} {\bf{Cancelling $[C0+]$ and $[C0-]$.}}}
  \end{figure}
\end{proof}

Generalizing the case of ``extended composites'', we can adapt the
definition from Definition~\ref{def:ExtendedComposite}, adding the
following additional types of terms:
\begin{itemize}
\item Nodes of type $[L0*]$ and $[R0*]$. 
As the notation suggests, terms of type $[L0*]$ are
have some node with no inputs (i.e. a $\mu^w_0$ for some $w>0$) 
feeding into a left-extended bottom vertex. Again, $*$ can be $-$, $g$, or $+$,
as in the conventions of Definition~\ref{def:ExtendedComposite}.
\end{itemize}

\begin{lemma}
  \label{lem:CancelL2}
  Terms of type $(L2)$ cancel aaginst terms of type
  $(2L-)\cup [L0*]
  [L**]\setminus([L0+]\cup [LL-]\cup[LR+])$.  Similarly, terms of
  type $(R2)$ cancel against terms of type
  $(2R+)\cup[R**]\setminus([R0-]\cup [RL-]\cup[RR+])$.
\end{lemma}  
\begin{proof}
  This proof is the same as Lemma~\ref{lem:CancelL2w0},
  except now the pulling out operation may not disconnect the graph.
  For example, starting from a term of type $(L2)$, when pulling out
  results in a connected graph, the corresponding term is of type
  $[L0*]\setminus[L0+]$.
\end{proof}

Lemma~\ref{lem:RLmLRpw0} has the following analogue:

\begin{lemma}
  \label{lem:RLmLRp}
  Terms of type $[RL-]$ cancel against terms of type $[LR+]$.
\end{lemma}

\begin{proof}
  The proof of Lemma~\ref{lem:RLmLRpw0} applies.
\end{proof}

\begin{lemma}
  \label{lem:CancelL0pR0m}
  Terms of type $[L0+]$ cancel against terms of type $[R0-]$.
\end{lemma}

\begin{proof}
  Like in the proof of Lemma~\ref{lem:CLpCLm}, the cancellation is 
  obtained by moving the root vertex. In this case, though, there are additional
  extending vertices that are moved, as well. See Figure~\ref{fig:L0pR0m}
  for an illustration of the relation
  \begin{align*}
    \mu_{10}(&R_3,U_4,U_3,U_4,L_3,L_2,U_1,U_2,U_1,\mu_0^{\eVec{2}}) \\
    &= L_2 t^2 \\
    &=\mu_{10}(\mu_0^{\eVec{2}},R_3,U_4,U_3,U_4,L_3,L_2,U_1,U_2,U_1).
  \end{align*}
  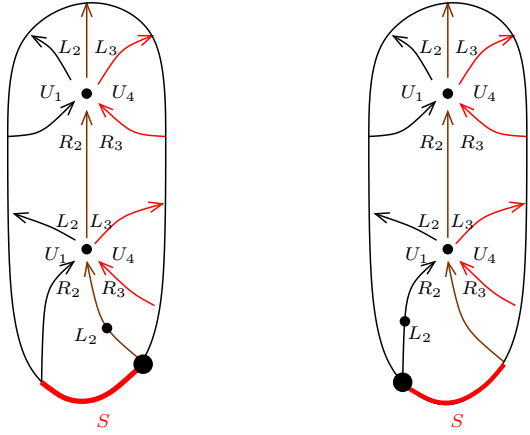
\begin{figure}[ht]
    \input{L0pR0m.pstex_t}
    \caption{\label{fig:L0pR0m} {\bf{Cancelling $[L0+]$ and $[R0-]$.}}}
  \end{figure}
\end{proof}

Lemma~\ref{lem:LLmRRpw0} has the following analogue. As in
Lemma~\ref{lem:LLmRRpw0}, a more precise cancellation scheme than the
one stated in the lemma can be found in the proof.

\begin{lemma}
  \label{lem:LLmRRp}
  Let $a$ be the output element to an $\Ainfty$ relation 
  with input sequence $a_1,\dots,a_n$ with $a_i\cdot a_{i+1}=0$ for
  all $i=1,\dots,n-1$ and weight $0$. For terms of weight $0$, we have cancellations of terms the following types:
  \[ [LL-], [L0-], [RR+], [R0+], (2C-), (2C+), (2L+), (2R-).\]
\end{lemma}

\begin{proof}

  The cancellation is similar to the one from Lemma~\ref{lem:LLmRRpw0}. Specifically, the cancellation is given in the following types of pairs: \\
  \begin{tabular}{r||c|c|c}
      & $|a|<|a_1|$ & $|a|=|a_1|$ & $|a|>|a_1|$ \\ 
  \hline 
  \hline
  $|a|<|a_n|$ & $[LL-]$
  $\Leftrightarrow$ $[RR+]$
            & $(2C-)$ $\Leftrightarrow$ $[RR+]$
              & $(2R-)$ $\Leftrightarrow$ $[RR+]$  \\
& $\cup[L0-]$ \qquad $\cup[R0+]$ & \qquad \qquad\qquad$\cup[R0+]$ & \qquad \qquad\qquad$\cup [R0+]$  \\
\hline 
  $|a|=|a_n|$ & $[LL-]$ 
$\Leftrightarrow$ $(2C+)$
            & $(2C-)$ $\Leftrightarrow$ $(2C+)$
              & $(2R-)$ $\Leftrightarrow$ $(2C+)$\\
& $\cup[L0-]$ \qquad \qquad\qquad & &   \\
\hline
  $|a|>|a_n|$ & $[LL-]$ 
  $\Leftrightarrow$ $(2L+)$
            & $(2C-)$ $\Leftrightarrow$ $(2L+)$
& $(2R-)$ $\Leftrightarrow$ $(2L+)$ \\
& $\cup[L0-]$\qquad\qquad\qquad & &  \\
\end{tabular}
  
  The involution realizing these cancellations are defined analogously
  to the proof of Lemma~\ref{lem:LLmRRpw0}. Roughly, we push in a edge
  and then pull out another edge. The key difference now is that,
  when one pulls out an edge in a weighted graph, the 
  graph is not necessarily disconnected: rather, one may cause a $\mu_0$ term
  to appear. Dually, one need not push in an edge for a two-component graph:
  rather, one can push in an edge corresponding to a term in $\mu_0$.
  This introduces the terms of type $[L0-]$ and $[R0+]$.

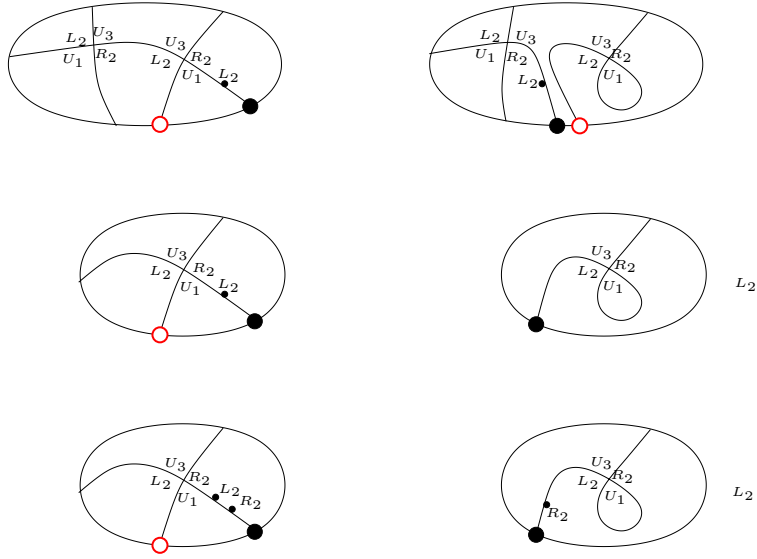
\begin{figure}[ht]
\input{L0mEx.pstex_t}
\caption{\label{fig:L0mEx} {\bf{Examples of cancellation of terms of type $[L0-]$, as in Lemma~\ref{lem:LLmRRp}.}}}
\end{figure}

  In Figure~\ref{fig:L0mEx}, we have illustrated cases of cancellations with terms in the first column
  of the table, with terms of the form $[L0-]$.
  The three lines in the figure exhibit the following three relations, respectively:
  \begin{align*} \mu_6(U_2,U_3^2,L_2,U_1,U_2,\mu_0^{\eVec{1}})
    & =\mu_4(\mu_2^{\eVec{1}}(U_2,U_3^2),L_2,U_1,U_2) \\
    \mu_4(U_2,U_3,L_2,\mu_0^{\eVec{1}})&=
    \mu_2(\mu_2^{\eVec{1}}(U_2,U_3),L_2) \\
    \mu_4(U_2R_2,U_3,L_2,\mu_0^{\eVec{1}})&=
    \mu_2(\mu_2^{\eVec{1}}(U_2R_2,U_3),L_2).
    \end{align*}
\end{proof}

\begin{proof}[of Theorem~\ref{thm:unWeightedAinftyAlgebra}]
In view of the above lemmas, Theorem~\ref{thm:WeightedAinftyAlgebra}
follows from the same reasoning as in the unweighted case (i.e. in the
proof of Theorem~\ref{thm:unWeightedAinftyAlgebra} above).

In fact, by Theorem~\ref{thm:unWeightedAinftyAlgebra}, it suffices to consider
$\Ainfty$ relations 
with incoming sequence of algebra
elements $(a_1,\dots,a_n)$ and weight vector $\vec{w}=(w_1,\dots,w_m)$ satisfies:
\begin{itemize}
\item  $|a_i|>0$ for all $i$,
\item non-zero weight vector $\vec{w}$.
\end{itemize}

For the weighted case, each non-trivial
$\Ainfty$ relation corresponds to a sequence $(a_1,\dots,a_n)$ of
incoming algebra elements (with $|a_i|>0$) and a weight vector
$\vec{w}$.

Again,  we measure the complexity by the number $\ell$ of
consecutive pairs of inputs with non-trivial $\mu_2$ (understood now
as $\mu_2^{0}$).  By Lemma~\ref{lem:NonMultiplyable}, the non-trivial $\Ainfty$ relations have $\ell\leq 1$.

Consider the cases where $\ell=1$ and the output is a power of $t$.
Lemma~\ref{lem:CancelC2} gives the cancellations of terms of type
$(C2*)$ with those in $[C**]\setminus ([CL+]\cup[CR-]\cup[C0+]\cup[C0-])$.
Again, these excluded terms have $\ell=0$.
When the output is not just a power of $t$, Lemma~\ref{lem:CancelL2}
cancels terms of type $(L2)$ and $(R2)$ against terms of type
\[ [L**]\cup[R**]\setminus ([L0+]\cup[LL-]\cup[LR+]\cup[R0-]\cup[RL-]\cup[RR+]).
\]

In the case where $\ell=0$ and the output is a power of $t$,
Lemma~\ref{lem:CLpCLm} now gives the cancellations
of the terms in $[CL+]\cup[CR-]\cup [C0+]\cup[C0-]$. 

Lemma~\ref{lem:CancelL0pR0m} cancels
$[L0+]$ against $[R0-]$; while Lemma~\ref{lem:RLmLRp}
gives the cancellation of $[RL-]$ and $[LR+]$.
The remaining cancellations are provided by Lemma~\ref{lem:LLmRRp}.
\end{proof}

%% file: PullingOutEx.pstex_t
\begin{picture}(0,0)%
\includegraphics{PullingOutEx.pstex}%
\end{picture}%
\setlength{\unitlength}{1243sp}%
\begingroup\makeatletter\ifx\SetFigFont\undefined%
\gdef\SetFigFont#1#2#3#4#5{%
  \reset@font\fontsize{#1}{#2pt}%
  \fontfamily{#3}\fontseries{#4}\fontshape{#5}%
  \selectfont}%
\fi\endgroup%
\begin{picture}(12984,7190)(3409,-5093)
\put(4366,749){\makebox(0,0)[lb]{\smash{{\SetFigFont{7}{8.4}{\rmdefault}{\mddefault}{\updefault}{\color[rgb]{0,0,0}$U_1$}%
}}}}
\put(4726,299){\makebox(0,0)[lb]{\smash{{\SetFigFont{7}{8.4}{\rmdefault}{\mddefault}{\updefault}{\color[rgb]{0,0,0}$R_2$}%
}}}}
\put(7516,1109){\makebox(0,0)[lb]{\smash{{\SetFigFont{7}{8.4}{\rmdefault}{\mddefault}{\updefault}{\color[rgb]{0,0,0}$L_2$}%
}}}}
\put(8146,1109){\makebox(0,0)[lb]{\smash{{\SetFigFont{7}{8.4}{\rmdefault}{\mddefault}{\updefault}{\color[rgb]{0,0,0}$U_3$}%
}}}}
\put(8146,659){\makebox(0,0)[lb]{\smash{{\SetFigFont{7}{8.4}{\rmdefault}{\mddefault}{\updefault}{\color[rgb]{0,0,0}$R_2$}%
}}}}
\put(11566,299){\makebox(0,0)[lb]{\smash{{\SetFigFont{7}{8.4}{\rmdefault}{\mddefault}{\updefault}{\color[rgb]{0,0,0}$U_1$}%
}}}}
\put(15346,659){\makebox(0,0)[lb]{\smash{{\SetFigFont{7}{8.4}{\rmdefault}{\mddefault}{\updefault}{\color[rgb]{0,0,0}$U_3$}%
}}}}
\put(5221,884){\makebox(0,0)[lb]{\smash{{\SetFigFont{7}{8.4}{\rmdefault}{\mddefault}{\updefault}{\color[rgb]{0,0,0}$U_3$}%
}}}}
\put(4591,1199){\makebox(0,0)[lb]{\smash{{\SetFigFont{7}{8.4}{\rmdefault}{\mddefault}{\updefault}{\color[rgb]{0,0,0}$L_2$}%
}}}}
\put(11836,839){\makebox(0,0)[lb]{\smash{{\SetFigFont{7}{8.4}{\rmdefault}{\mddefault}{\updefault}{\color[rgb]{0,0,0}$L_2$}%
}}}}
\put(12286,389){\makebox(0,0)[lb]{\smash{{\SetFigFont{7}{8.4}{\rmdefault}{\mddefault}{\updefault}{\color[rgb]{0,0,0}$U_3$}%
}}}}
\put(12016,-61){\makebox(0,0)[lb]{\smash{{\SetFigFont{7}{8.4}{\rmdefault}{\mddefault}{\updefault}{\color[rgb]{0,0,0}$R_2$}%
}}}}
\put(13366,-241){\makebox(0,0)[lb]{\smash{{\SetFigFont{7}{8.4}{\rmdefault}{\mddefault}{\updefault}{\color[rgb]{0,0,0}$L_2$}%
}}}}
\put(14131,-241){\makebox(0,0)[lb]{\smash{{\SetFigFont{7}{8.4}{\rmdefault}{\mddefault}{\updefault}{\color[rgb]{0,0,0}$U_1$}%
}}}}
\put(14086,344){\makebox(0,0)[lb]{\smash{{\SetFigFont{7}{8.4}{\rmdefault}{\mddefault}{\updefault}{\color[rgb]{0,0,0}$R_2$}%
}}}}
\put(14806,-106){\makebox(0,0)[lb]{\smash{{\SetFigFont{7}{8.4}{\rmdefault}{\mddefault}{\updefault}{\color[rgb]{0,0,0}$U_1$}%
}}}}
\put(15391,119){\makebox(0,0)[lb]{\smash{{\SetFigFont{7}{8.4}{\rmdefault}{\mddefault}{\updefault}{\color[rgb]{0,0,0}$R_2$}%
}}}}
\put(14671,479){\makebox(0,0)[lb]{\smash{{\SetFigFont{7}{8.4}{\rmdefault}{\mddefault}{\updefault}{\color[rgb]{0,0,0}$L_2$}%
}}}}
\put(6076,209){\makebox(0,0)[lb]{\smash{{\SetFigFont{7}{8.4}{\rmdefault}{\mddefault}{\updefault}{\color[rgb]{0,0,0}$L_2$}%
}}}}
\put(6751,344){\makebox(0,0)[lb]{\smash{{\SetFigFont{7}{8.4}{\rmdefault}{\mddefault}{\updefault}{\color[rgb]{0,0,0}$U_1$}%
}}}}
\put(6886,794){\makebox(0,0)[lb]{\smash{{\SetFigFont{7}{8.4}{\rmdefault}{\mddefault}{\updefault}{\color[rgb]{0,0,0}$R_2$}%
}}}}
\put(6121,794){\makebox(0,0)[lb]{\smash{{\SetFigFont{7}{8.4}{\rmdefault}{\mddefault}{\updefault}{\color[rgb]{0,0,0}$U_3$}%
}}}}
\put(7651,524){\makebox(0,0)[lb]{\smash{{\SetFigFont{7}{8.4}{\rmdefault}{\mddefault}{\updefault}{\color[rgb]{0,0,0}$U_1$}%
}}}}
\put(13636,209){\makebox(0,0)[lb]{\smash{{\SetFigFont{7}{8.4}{\rmdefault}{\mddefault}{\updefault}{\color[rgb]{0,0,0}$U_3$}%
}}}}
\put(4366,-3751){\makebox(0,0)[lb]{\smash{{\SetFigFont{7}{8.4}{\rmdefault}{\mddefault}{\updefault}{\color[rgb]{0,0,0}$U_1$}%
}}}}
\put(4771,-3301){\makebox(0,0)[lb]{\smash{{\SetFigFont{7}{8.4}{\rmdefault}{\mddefault}{\updefault}{\color[rgb]{0,0,0}$L_2$}%
}}}}
\put(6436,-4291){\makebox(0,0)[lb]{\smash{{\SetFigFont{7}{8.4}{\rmdefault}{\mddefault}{\updefault}{\color[rgb]{0,0,0}$L_2$}%
}}}}
\put(6841,-3571){\makebox(0,0)[lb]{\smash{{\SetFigFont{7}{8.4}{\rmdefault}{\mddefault}{\updefault}{\color[rgb]{0,0,0}$R_2$}%
}}}}
\put(7021,-4111){\makebox(0,0)[lb]{\smash{{\SetFigFont{7}{8.4}{\rmdefault}{\mddefault}{\updefault}{\color[rgb]{0,0,0}$U_1$}%
}}}}
\put(7651,-3886){\makebox(0,0)[lb]{\smash{{\SetFigFont{7}{8.4}{\rmdefault}{\mddefault}{\updefault}{\color[rgb]{0,0,0}$U_1$}%
}}}}
\put(7516,-3391){\makebox(0,0)[lb]{\smash{{\SetFigFont{7}{8.4}{\rmdefault}{\mddefault}{\updefault}{\color[rgb]{0,0,0}$L_2$}%
}}}}
\put(8146,-3391){\makebox(0,0)[lb]{\smash{{\SetFigFont{7}{8.4}{\rmdefault}{\mddefault}{\updefault}{\color[rgb]{0,0,0}$U_3$}%
}}}}
\put(8146,-3841){\makebox(0,0)[lb]{\smash{{\SetFigFont{7}{8.4}{\rmdefault}{\mddefault}{\updefault}{\color[rgb]{0,0,0}$R_2$}%
}}}}
\put(11566,-3751){\makebox(0,0)[lb]{\smash{{\SetFigFont{7}{8.4}{\rmdefault}{\mddefault}{\updefault}{\color[rgb]{0,0,0}$U_1$}%
}}}}
\put(11971,-3301){\makebox(0,0)[lb]{\smash{{\SetFigFont{7}{8.4}{\rmdefault}{\mddefault}{\updefault}{\color[rgb]{0,0,0}$L_2$}%
}}}}
\put(12466,-3841){\makebox(0,0)[lb]{\smash{{\SetFigFont{7}{8.4}{\rmdefault}{\mddefault}{\updefault}{\color[rgb]{0,0,0}$U_3$}%
}}}}
\put(11926,-4201){\makebox(0,0)[lb]{\smash{{\SetFigFont{7}{8.4}{\rmdefault}{\mddefault}{\updefault}{\color[rgb]{0,0,0}$R_2$}%
}}}}
\put(13411,-3841){\makebox(0,0)[lb]{\smash{{\SetFigFont{7}{8.4}{\rmdefault}{\mddefault}{\updefault}{\color[rgb]{0,0,0}$U_3$}%
}}}}
\put(13771,-4336){\makebox(0,0)[lb]{\smash{{\SetFigFont{7}{8.4}{\rmdefault}{\mddefault}{\updefault}{\color[rgb]{0,0,0}$L_2$}%
}}}}
\put(13996,-3526){\makebox(0,0)[lb]{\smash{{\SetFigFont{7}{8.4}{\rmdefault}{\mddefault}{\updefault}{\color[rgb]{0,0,0}$R_2$}%
}}}}
\put(14266,-4021){\makebox(0,0)[lb]{\smash{{\SetFigFont{7}{8.4}{\rmdefault}{\mddefault}{\updefault}{\color[rgb]{0,0,0}$U_1$}%
}}}}
\put(14851,-3886){\makebox(0,0)[lb]{\smash{{\SetFigFont{7}{8.4}{\rmdefault}{\mddefault}{\updefault}{\color[rgb]{0,0,0}$U_1$}%
}}}}
\put(14716,-3391){\makebox(0,0)[lb]{\smash{{\SetFigFont{7}{8.4}{\rmdefault}{\mddefault}{\updefault}{\color[rgb]{0,0,0}$L_2$}%
}}}}
\put(15346,-3391){\makebox(0,0)[lb]{\smash{{\SetFigFont{7}{8.4}{\rmdefault}{\mddefault}{\updefault}{\color[rgb]{0,0,0}$U_3$}%
}}}}
\put(15346,-3841){\makebox(0,0)[lb]{\smash{{\SetFigFont{7}{8.4}{\rmdefault}{\mddefault}{\updefault}{\color[rgb]{0,0,0}$R_2$}%
}}}}
\put(5221,-3661){\makebox(0,0)[lb]{\smash{{\SetFigFont{7}{8.4}{\rmdefault}{\mddefault}{\updefault}{\color[rgb]{0,0,0}$U_3$}%
}}}}
\put(6166,-3706){\makebox(0,0)[lb]{\smash{{\SetFigFont{7}{8.4}{\rmdefault}{\mddefault}{\updefault}{\color[rgb]{0,0,0}$U_3$}%
}}}}
\put(4906,-4111){\makebox(0,0)[lb]{\smash{{\SetFigFont{7}{8.4}{\rmdefault}{\mddefault}{\updefault}{\color[rgb]{0,0,0}$R_2$}%
}}}}
\end{picture}%

%% file: PullingOut.pstex_t
\begin{picture}(0,0)%
\includegraphics{PullingOut.pstex}%
\end{picture}%
\setlength{\unitlength}{1243sp}%
\begingroup\makeatletter\ifx\SetFigFont\undefined%
\gdef\SetFigFont#1#2#3#4#5{%
  \reset@font\fontsize{#1}{#2pt}%
  \fontfamily{#3}\fontseries{#4}\fontshape{#5}%
  \selectfont}%
\fi\endgroup%
\begin{picture}(13884,3072)(2509,-5440)
\put(14761,-3841){\makebox(0,0)[lb]{\smash{{\SetFigFont{7}{8.4}{\rmdefault}{\mddefault}{\updefault}{\color[rgb]{0,0,0}$\Gamma_2$}%
}}}}
\put(3601,-3661){\makebox(0,0)[lb]{\smash{{\SetFigFont{7}{8.4}{\rmdefault}{\mddefault}{\updefault}{\color[rgb]{0,0,0}$\Gamma_1$}%
}}}}
\put(6661,-3841){\makebox(0,0)[lb]{\smash{{\SetFigFont{7}{8.4}{\rmdefault}{\mddefault}{\updefault}{\color[rgb]{0,0,0}$\Gamma_2$}%
}}}}
\put(11701,-3661){\makebox(0,0)[lb]{\smash{{\SetFigFont{7}{8.4}{\rmdefault}{\mddefault}{\updefault}{\color[rgb]{0,0,0}$\Gamma_1$}%
}}}}
\put(4996,-4381){\makebox(0,0)[lb]{\smash{{\SetFigFont{7}{8.4}{\rmdefault}{\mddefault}{\updefault}{\color[rgb]{0,0,0}$P$}%
}}}}
\put(13006,-5326){\makebox(0,0)[lb]{\smash{{\SetFigFont{7}{8.4}{\rmdefault}{\mddefault}{\updefault}{\color[rgb]{1,0,0}$S$}%
}}}}
\end{picture}%

%% file: CLpCRm.pstex_t
\begin{picture}(0,0)%
\includegraphics{CLpCRm.pstex}%
\end{picture}%
\setlength{\unitlength}{1243sp}%
\begingroup\makeatletter\ifx\SetFigFont\undefined%
\gdef\SetFigFont#1#2#3#4#5{%
  \reset@font\fontsize{#1}{#2pt}%
  \fontfamily{#3}\fontseries{#4}\fontshape{#5}%
  \selectfont}%
\fi\endgroup%
\begin{picture}(12984,2716)(3409,-5084)
\put(4501,-3661){\makebox(0,0)[lb]{\smash{{\SetFigFont{7}{8.4}{\rmdefault}{\mddefault}{\updefault}{\color[rgb]{0,0,0}$\Gamma_2$}%
}}}}
\put(7426,-3796){\makebox(0,0)[lb]{\smash{{\SetFigFont{7}{8.4}{\rmdefault}{\mddefault}{\updefault}{\color[rgb]{0,0,0}$\Gamma_1$}%
}}}}
\put(11701,-3661){\makebox(0,0)[lb]{\smash{{\SetFigFont{7}{8.4}{\rmdefault}{\mddefault}{\updefault}{\color[rgb]{0,0,0}$\Gamma_1$}%
}}}}
\put(14626,-3796){\makebox(0,0)[lb]{\smash{{\SetFigFont{7}{8.4}{\rmdefault}{\mddefault}{\updefault}{\color[rgb]{0,0,0}$\Gamma_2$}%
}}}}
\end{picture}%

%% file: CLpCRmEx.pstex_t
\begin{picture}(0,0)%
\includegraphics{CLpCRmEx.pstex}%
\end{picture}%
\setlength{\unitlength}{1243sp}%
\begingroup\makeatletter\ifx\SetFigFont\undefined%
\gdef\SetFigFont#1#2#3#4#5{%
  \reset@font\fontsize{#1}{#2pt}%
  \fontfamily{#3}\fontseries{#4}\fontshape{#5}%
  \selectfont}%
\fi\endgroup%
\begin{picture}(12984,2777)(3409,-5145)
\put(6211,-3841){\makebox(0,0)[lb]{\smash{{\SetFigFont{7}{8.4}{\rmdefault}{\mddefault}{\updefault}{\color[rgb]{0,0,0}$U_3$}%
}}}}
\put(6571,-4381){\makebox(0,0)[lb]{\smash{{\SetFigFont{7}{8.4}{\rmdefault}{\mddefault}{\updefault}{\color[rgb]{0,0,0}$L_2$}%
}}}}
\put(6751,-3481){\makebox(0,0)[lb]{\smash{{\SetFigFont{7}{8.4}{\rmdefault}{\mddefault}{\updefault}{\color[rgb]{0,0,0}$R_2$}%
}}}}
\put(6841,-4066){\makebox(0,0)[lb]{\smash{{\SetFigFont{7}{8.4}{\rmdefault}{\mddefault}{\updefault}{\color[rgb]{0,0,0}$U_1$}%
}}}}
\put(7651,-3886){\makebox(0,0)[lb]{\smash{{\SetFigFont{7}{8.4}{\rmdefault}{\mddefault}{\updefault}{\color[rgb]{0,0,0}$U_1$}%
}}}}
\put(7516,-3391){\makebox(0,0)[lb]{\smash{{\SetFigFont{7}{8.4}{\rmdefault}{\mddefault}{\updefault}{\color[rgb]{0,0,0}$L_2$}%
}}}}
\put(8146,-3391){\makebox(0,0)[lb]{\smash{{\SetFigFont{7}{8.4}{\rmdefault}{\mddefault}{\updefault}{\color[rgb]{0,0,0}$U_3$}%
}}}}
\put(8146,-3841){\makebox(0,0)[lb]{\smash{{\SetFigFont{7}{8.4}{\rmdefault}{\mddefault}{\updefault}{\color[rgb]{0,0,0}$R_2$}%
}}}}
\put(11566,-3751){\makebox(0,0)[lb]{\smash{{\SetFigFont{7}{8.4}{\rmdefault}{\mddefault}{\updefault}{\color[rgb]{0,0,0}$U_1$}%
}}}}
\put(11971,-3301){\makebox(0,0)[lb]{\smash{{\SetFigFont{7}{8.4}{\rmdefault}{\mddefault}{\updefault}{\color[rgb]{0,0,0}$L_2$}%
}}}}
\put(12466,-3841){\makebox(0,0)[lb]{\smash{{\SetFigFont{7}{8.4}{\rmdefault}{\mddefault}{\updefault}{\color[rgb]{0,0,0}$U_3$}%
}}}}
\put(11926,-4201){\makebox(0,0)[lb]{\smash{{\SetFigFont{7}{8.4}{\rmdefault}{\mddefault}{\updefault}{\color[rgb]{0,0,0}$R_2$}%
}}}}
\put(13411,-3841){\makebox(0,0)[lb]{\smash{{\SetFigFont{7}{8.4}{\rmdefault}{\mddefault}{\updefault}{\color[rgb]{0,0,0}$U_3$}%
}}}}
\put(4366,-3751){\makebox(0,0)[lb]{\smash{{\SetFigFont{7}{8.4}{\rmdefault}{\mddefault}{\updefault}{\color[rgb]{0,0,0}$U_1$}%
}}}}
\put(4771,-3301){\makebox(0,0)[lb]{\smash{{\SetFigFont{7}{8.4}{\rmdefault}{\mddefault}{\updefault}{\color[rgb]{0,0,0}$L_2$}%
}}}}
\put(5266,-3841){\makebox(0,0)[lb]{\smash{{\SetFigFont{7}{8.4}{\rmdefault}{\mddefault}{\updefault}{\color[rgb]{0,0,0}$U_3$}%
}}}}
\put(4726,-4201){\makebox(0,0)[lb]{\smash{{\SetFigFont{7}{8.4}{\rmdefault}{\mddefault}{\updefault}{\color[rgb]{0,0,0}$R_2$}%
}}}}
\put(15346,-3391){\makebox(0,0)[lb]{\smash{{\SetFigFont{7}{8.4}{\rmdefault}{\mddefault}{\updefault}{\color[rgb]{0,0,0}$U_3$}%
}}}}
\put(15346,-3841){\makebox(0,0)[lb]{\smash{{\SetFigFont{7}{8.4}{\rmdefault}{\mddefault}{\updefault}{\color[rgb]{0,0,0}$R_2$}%
}}}}
\put(13771,-4381){\makebox(0,0)[lb]{\smash{{\SetFigFont{7}{8.4}{\rmdefault}{\mddefault}{\updefault}{\color[rgb]{0,0,0}$L_2$}%
}}}}
\put(13951,-3481){\makebox(0,0)[lb]{\smash{{\SetFigFont{7}{8.4}{\rmdefault}{\mddefault}{\updefault}{\color[rgb]{0,0,0}$R_2$}%
}}}}
\put(14041,-4066){\makebox(0,0)[lb]{\smash{{\SetFigFont{7}{8.4}{\rmdefault}{\mddefault}{\updefault}{\color[rgb]{0,0,0}$U_1$}%
}}}}
\put(14851,-3886){\makebox(0,0)[lb]{\smash{{\SetFigFont{7}{8.4}{\rmdefault}{\mddefault}{\updefault}{\color[rgb]{0,0,0}$U_1$}%
}}}}
\put(14716,-3391){\makebox(0,0)[lb]{\smash{{\SetFigFont{7}{8.4}{\rmdefault}{\mddefault}{\updefault}{\color[rgb]{0,0,0}$L_2$}%
}}}}
\end{picture}%

%% file: LLm.pstex_t
\begin{picture}(0,0)%
\includegraphics{LLm.pstex}%
\end{picture}%
\setlength{\unitlength}{829sp}%
\begingroup\makeatletter\ifx\SetFigFont\undefined%
\gdef\SetFigFont#1#2#3#4#5{%
  \reset@font\fontsize{#1}{#2pt}%
  \fontfamily{#3}\fontseries{#4}\fontshape{#5}%
  \selectfont}%
\fi\endgroup%
\begin{picture}(27069,11956)(-11711,-12434)
\end{picture}%

%% file: LLmEx.pstex_t
\begin{picture}(0,0)%
\includegraphics{LLmEx.pstex}%
\end{picture}%
\setlength{\unitlength}{829sp}%
\begingroup\makeatletter\ifx\SetFigFont\undefined%
\gdef\SetFigFont#1#2#3#4#5{%
  \reset@font\fontsize{#1}{#2pt}%
  \fontfamily{#3}\fontseries{#4}\fontshape{#5}%
  \selectfont}%
\fi\endgroup%
\begin{picture}(26999,15686)(-5383,-17244)
\put(2701,-15091){\rotatebox{360.0}{\makebox(0,0)[rb]{\smash{{\SetFigFont{5}{6.0}{\rmdefault}{\mddefault}{\updefault}{\color[rgb]{0,0,0}$U_1$}%
}}}}}
\put(1981,-14551){\rotatebox{360.0}{\makebox(0,0)[rb]{\smash{{\SetFigFont{5}{6.0}{\rmdefault}{\mddefault}{\updefault}{\color[rgb]{0,0,0}$R_2$}%
}}}}}
\put(1621,-15091){\rotatebox{360.0}{\makebox(0,0)[rb]{\smash{{\SetFigFont{5}{6.0}{\rmdefault}{\mddefault}{\updefault}{\color[rgb]{0,0,0}$U_3$}%
}}}}}
\put(2161,-15541){\rotatebox{360.0}{\makebox(0,0)[rb]{\smash{{\SetFigFont{5}{6.0}{\rmdefault}{\mddefault}{\updefault}{\color[rgb]{0,0,0}$L_2$}%
}}}}}
\put(-1304,-15901){\rotatebox{360.0}{\makebox(0,0)[rb]{\smash{{\SetFigFont{5}{6.0}{\rmdefault}{\mddefault}{\updefault}{\color[rgb]{0,0,0}$R_2$}%
}}}}}
\put(-764,-15406){\rotatebox{360.0}{\makebox(0,0)[rb]{\smash{{\SetFigFont{5}{6.0}{\rmdefault}{\mddefault}{\updefault}{\color[rgb]{0,0,0}$U_3$}%
}}}}}
\put(-1799,-15271){\rotatebox{360.0}{\makebox(0,0)[rb]{\smash{{\SetFigFont{5}{6.0}{\rmdefault}{\mddefault}{\updefault}{\color[rgb]{0,0,0}$U_1$}%
}}}}}
\put(-1034,-14776){\rotatebox{360.0}{\makebox(0,0)[rb]{\smash{{\SetFigFont{5}{6.0}{\rmdefault}{\mddefault}{\updefault}{\color[rgb]{0,0,0}$L_2$}%
}}}}}
\put(1981,-16081){\rotatebox{360.0}{\makebox(0,0)[rb]{\smash{{\SetFigFont{5}{6.0}{\rmdefault}{\mddefault}{\updefault}{\color[rgb]{0,0,0}$R_2$}%
}}}}}
\put(1891,-16576){\rotatebox{360.0}{\makebox(0,0)[rb]{\smash{{\SetFigFont{5}{6.0}{\rmdefault}{\mddefault}{\updefault}{\color[rgb]{0,0,0}$L_2$}%
}}}}}
\put(2701,-9241){\rotatebox{360.0}{\makebox(0,0)[rb]{\smash{{\SetFigFont{5}{6.0}{\rmdefault}{\mddefault}{\updefault}{\color[rgb]{0,0,0}$U_1$}%
}}}}}
\put(1981,-8701){\rotatebox{360.0}{\makebox(0,0)[rb]{\smash{{\SetFigFont{5}{6.0}{\rmdefault}{\mddefault}{\updefault}{\color[rgb]{0,0,0}$R_2$}%
}}}}}
\put(1621,-9241){\rotatebox{360.0}{\makebox(0,0)[rb]{\smash{{\SetFigFont{5}{6.0}{\rmdefault}{\mddefault}{\updefault}{\color[rgb]{0,0,0}$U_3$}%
}}}}}
\put(2161,-9691){\rotatebox{360.0}{\makebox(0,0)[rb]{\smash{{\SetFigFont{5}{6.0}{\rmdefault}{\mddefault}{\updefault}{\color[rgb]{0,0,0}$L_2$}%
}}}}}
\put(1801,-10591){\rotatebox{360.0}{\makebox(0,0)[rb]{\smash{{\SetFigFont{5}{6.0}{\rmdefault}{\mddefault}{\updefault}{\color[rgb]{0,0,0}$R_2$}%
}}}}}
\put(-1304,-10051){\rotatebox{360.0}{\makebox(0,0)[rb]{\smash{{\SetFigFont{5}{6.0}{\rmdefault}{\mddefault}{\updefault}{\color[rgb]{0,0,0}$R_2$}%
}}}}}
\put(-764,-9556){\rotatebox{360.0}{\makebox(0,0)[rb]{\smash{{\SetFigFont{5}{6.0}{\rmdefault}{\mddefault}{\updefault}{\color[rgb]{0,0,0}$U_3$}%
}}}}}
\put(-1799,-9421){\rotatebox{360.0}{\makebox(0,0)[rb]{\smash{{\SetFigFont{5}{6.0}{\rmdefault}{\mddefault}{\updefault}{\color[rgb]{0,0,0}$U_1$}%
}}}}}
\put(-1034,-8926){\rotatebox{360.0}{\makebox(0,0)[rb]{\smash{{\SetFigFont{5}{6.0}{\rmdefault}{\mddefault}{\updefault}{\color[rgb]{0,0,0}$L_2$}%
}}}}}
\put(2701,-3391){\rotatebox{360.0}{\makebox(0,0)[rb]{\smash{{\SetFigFont{5}{6.0}{\rmdefault}{\mddefault}{\updefault}{\color[rgb]{0,0,0}$U_1$}%
}}}}}
\put(-269,-3211){\rotatebox{360.0}{\makebox(0,0)[rb]{\smash{{\SetFigFont{5}{6.0}{\rmdefault}{\mddefault}{\updefault}{\color[rgb]{0,0,0}$U_3$}%
}}}}}
\put(-539,-2581){\rotatebox{360.0}{\makebox(0,0)[rb]{\smash{{\SetFigFont{5}{6.0}{\rmdefault}{\mddefault}{\updefault}{\color[rgb]{0,0,0}$L_2$}%
}}}}}
\put(-1169,-2851){\rotatebox{360.0}{\makebox(0,0)[rb]{\smash{{\SetFigFont{5}{6.0}{\rmdefault}{\mddefault}{\updefault}{\color[rgb]{0,0,0}$U_1$}%
}}}}}
\put(-2249,-3481){\rotatebox{360.0}{\makebox(0,0)[rb]{\smash{{\SetFigFont{5}{6.0}{\rmdefault}{\mddefault}{\updefault}{\color[rgb]{0,0,0}$L_2$}%
}}}}}
\put(-2249,-2941){\rotatebox{360.0}{\makebox(0,0)[rb]{\smash{{\SetFigFont{5}{6.0}{\rmdefault}{\mddefault}{\updefault}{\color[rgb]{0,0,0}$U_1$}%
}}}}}
\put(1981,-2851){\rotatebox{360.0}{\makebox(0,0)[rb]{\smash{{\SetFigFont{5}{6.0}{\rmdefault}{\mddefault}{\updefault}{\color[rgb]{0,0,0}$R_2$}%
}}}}}
\put(-2789,-2581){\rotatebox{360.0}{\makebox(0,0)[rb]{\smash{{\SetFigFont{5}{6.0}{\rmdefault}{\mddefault}{\updefault}{\color[rgb]{0,0,0}$R_2$}%
}}}}}
\put(-2969,-3211){\rotatebox{360.0}{\makebox(0,0)[rb]{\smash{{\SetFigFont{5}{6.0}{\rmdefault}{\mddefault}{\updefault}{\color[rgb]{0,0,0}$U_3$}%
}}}}}
\put(1621,-3391){\rotatebox{360.0}{\makebox(0,0)[rb]{\smash{{\SetFigFont{5}{6.0}{\rmdefault}{\mddefault}{\updefault}{\color[rgb]{0,0,0}$U_3$}%
}}}}}
\put(2161,-3841){\rotatebox{360.0}{\makebox(0,0)[rb]{\smash{{\SetFigFont{5}{6.0}{\rmdefault}{\mddefault}{\updefault}{\color[rgb]{0,0,0}$L_2$}%
}}}}}
\put(-1169,-3391){\rotatebox{360.0}{\makebox(0,0)[rb]{\smash{{\SetFigFont{5}{6.0}{\rmdefault}{\mddefault}{\updefault}{\color[rgb]{0,0,0}$R_2$}%
}}}}}
\put(1801,-4741){\rotatebox{360.0}{\makebox(0,0)[rb]{\smash{{\SetFigFont{5}{6.0}{\rmdefault}{\mddefault}{\updefault}{\color[rgb]{0,0,0}$R_2$}%
}}}}}
\put(17191,-15451){\rotatebox{360.0}{\makebox(0,0)[rb]{\smash{{\SetFigFont{5}{6.0}{\rmdefault}{\mddefault}{\updefault}{\color[rgb]{0,0,0}$L_2$}%
}}}}}
\put(15121,-15721){\rotatebox{360.0}{\makebox(0,0)[rb]{\smash{{\SetFigFont{5}{6.0}{\rmdefault}{\mddefault}{\updefault}{\color[rgb]{0,0,0}$R_2$}%
}}}}}
\put(16741,-15001){\rotatebox{360.0}{\makebox(0,0)[rb]{\smash{{\SetFigFont{5}{6.0}{\rmdefault}{\mddefault}{\updefault}{\color[rgb]{0,0,0}$U_3$}%
}}}}}
\put(14671,-14731){\rotatebox{360.0}{\makebox(0,0)[rb]{\smash{{\SetFigFont{5}{6.0}{\rmdefault}{\mddefault}{\updefault}{\color[rgb]{0,0,0}$L_2$}%
}}}}}
\put(17641,-15001){\rotatebox{360.0}{\makebox(0,0)[rb]{\smash{{\SetFigFont{5}{6.0}{\rmdefault}{\mddefault}{\updefault}{\color[rgb]{0,0,0}$U_1$}%
}}}}}
\put(17281,-14551){\rotatebox{360.0}{\makebox(0,0)[rb]{\smash{{\SetFigFont{5}{6.0}{\rmdefault}{\mddefault}{\updefault}{\color[rgb]{0,0,0}$R_2$}%
}}}}}
\put(14401,-15361){\rotatebox{360.0}{\makebox(0,0)[rb]{\smash{{\SetFigFont{5}{6.0}{\rmdefault}{\mddefault}{\updefault}{\color[rgb]{0,0,0}$U_1$}%
}}}}}
\put(15301,-15181){\rotatebox{360.0}{\makebox(0,0)[rb]{\smash{{\SetFigFont{5}{6.0}{\rmdefault}{\mddefault}{\updefault}{\color[rgb]{0,0,0}$U_3$}%
}}}}}
\put(15076,-16486){\rotatebox{360.0}{\makebox(0,0)[rb]{\smash{{\SetFigFont{5}{6.0}{\rmdefault}{\mddefault}{\updefault}{\color[rgb]{0,0,0}$L_2$}%
}}}}}
\put(17191,-9601){\rotatebox{360.0}{\makebox(0,0)[rb]{\smash{{\SetFigFont{5}{6.0}{\rmdefault}{\mddefault}{\updefault}{\color[rgb]{0,0,0}$L_2$}%
}}}}}
\put(15121,-9871){\rotatebox{360.0}{\makebox(0,0)[rb]{\smash{{\SetFigFont{5}{6.0}{\rmdefault}{\mddefault}{\updefault}{\color[rgb]{0,0,0}$R_2$}%
}}}}}
\put(16741,-9151){\rotatebox{360.0}{\makebox(0,0)[rb]{\smash{{\SetFigFont{5}{6.0}{\rmdefault}{\mddefault}{\updefault}{\color[rgb]{0,0,0}$U_3$}%
}}}}}
\put(14671,-8881){\rotatebox{360.0}{\makebox(0,0)[rb]{\smash{{\SetFigFont{5}{6.0}{\rmdefault}{\mddefault}{\updefault}{\color[rgb]{0,0,0}$L_2$}%
}}}}}
\put(17641,-9151){\rotatebox{360.0}{\makebox(0,0)[rb]{\smash{{\SetFigFont{5}{6.0}{\rmdefault}{\mddefault}{\updefault}{\color[rgb]{0,0,0}$U_1$}%
}}}}}
\put(17281,-8701){\rotatebox{360.0}{\makebox(0,0)[rb]{\smash{{\SetFigFont{5}{6.0}{\rmdefault}{\mddefault}{\updefault}{\color[rgb]{0,0,0}$R_2$}%
}}}}}
\put(14401,-9511){\rotatebox{360.0}{\makebox(0,0)[rb]{\smash{{\SetFigFont{5}{6.0}{\rmdefault}{\mddefault}{\updefault}{\color[rgb]{0,0,0}$U_1$}%
}}}}}
\put(15301,-9331){\rotatebox{360.0}{\makebox(0,0)[rb]{\smash{{\SetFigFont{5}{6.0}{\rmdefault}{\mddefault}{\updefault}{\color[rgb]{0,0,0}$U_3$}%
}}}}}
\put(17191,-3751){\rotatebox{360.0}{\makebox(0,0)[rb]{\smash{{\SetFigFont{5}{6.0}{\rmdefault}{\mddefault}{\updefault}{\color[rgb]{0,0,0}$L_2$}%
}}}}}
\put(15121,-4021){\rotatebox{360.0}{\makebox(0,0)[rb]{\smash{{\SetFigFont{5}{6.0}{\rmdefault}{\mddefault}{\updefault}{\color[rgb]{0,0,0}$R_2$}%
}}}}}
\put(16741,-3301){\rotatebox{360.0}{\makebox(0,0)[rb]{\smash{{\SetFigFont{5}{6.0}{\rmdefault}{\mddefault}{\updefault}{\color[rgb]{0,0,0}$U_3$}%
}}}}}
\put(14671,-3031){\rotatebox{360.0}{\makebox(0,0)[rb]{\smash{{\SetFigFont{5}{6.0}{\rmdefault}{\mddefault}{\updefault}{\color[rgb]{0,0,0}$L_2$}%
}}}}}
\put(17641,-3301){\rotatebox{360.0}{\makebox(0,0)[rb]{\smash{{\SetFigFont{5}{6.0}{\rmdefault}{\mddefault}{\updefault}{\color[rgb]{0,0,0}$U_1$}%
}}}}}
\put(17281,-2851){\rotatebox{360.0}{\makebox(0,0)[rb]{\smash{{\SetFigFont{5}{6.0}{\rmdefault}{\mddefault}{\updefault}{\color[rgb]{0,0,0}$R_2$}%
}}}}}
\put(12241,-3121){\rotatebox{360.0}{\makebox(0,0)[rb]{\smash{{\SetFigFont{5}{6.0}{\rmdefault}{\mddefault}{\updefault}{\color[rgb]{0,0,0}$R_2$}%
}}}}}
\put(14401,-3661){\rotatebox{360.0}{\makebox(0,0)[rb]{\smash{{\SetFigFont{5}{6.0}{\rmdefault}{\mddefault}{\updefault}{\color[rgb]{0,0,0}$U_1$}%
}}}}}
\put(12691,-3661){\rotatebox{360.0}{\makebox(0,0)[rb]{\smash{{\SetFigFont{5}{6.0}{\rmdefault}{\mddefault}{\updefault}{\color[rgb]{0,0,0}$U_1$}%
}}}}}
\put(15301,-3481){\rotatebox{360.0}{\makebox(0,0)[rb]{\smash{{\SetFigFont{5}{6.0}{\rmdefault}{\mddefault}{\updefault}{\color[rgb]{0,0,0}$U_3$}%
}}}}}
\put(11881,-3661){\rotatebox{360.0}{\makebox(0,0)[rb]{\smash{{\SetFigFont{5}{6.0}{\rmdefault}{\mddefault}{\updefault}{\color[rgb]{0,0,0}$U_3$}%
}}}}}
\put(12331,-4201){\rotatebox{360.0}{\makebox(0,0)[rb]{\smash{{\SetFigFont{5}{6.0}{\rmdefault}{\mddefault}{\updefault}{\color[rgb]{0,0,0}$L_2$}%
}}}}}
\put(12241,-4831){\rotatebox{360.0}{\makebox(0,0)[rb]{\smash{{\SetFigFont{5}{6.0}{\rmdefault}{\mddefault}{\updefault}{\color[rgb]{0,0,0}$R_2$}%
}}}}}
\put(21601,-9061){\rotatebox{360.0}{\makebox(0,0)[rb]{\smash{{\SetFigFont{5}{6.0}{\rmdefault}{\mddefault}{\updefault}{\color[rgb]{0,0,0}$R_2$}%
}}}}}
\put(21601,-14911){\rotatebox{360.0}{\makebox(0,0)[rb]{\smash{{\SetFigFont{5}{6.0}{\rmdefault}{\mddefault}{\updefault}{\color[rgb]{0,0,0}$R_2$}%
}}}}}
\end{picture}%

%% file: MoveRoot.pstex_t
\begin{picture}(0,0)%
\includegraphics{MoveRoot.pstex}%
\end{picture}%
\setlength{\unitlength}{1243sp}%
\begingroup\makeatletter\ifx\SetFigFont\undefined%
\gdef\SetFigFont#1#2#3#4#5{%
  \reset@font\fontsize{#1}{#2pt}%
  \fontfamily{#3}\fontseries{#4}\fontshape{#5}%
  \selectfont}%
\fi\endgroup%
\begin{picture}(10605,8056)(11686,-5159)
\put(22276,-961){\makebox(0,0)[lb]{\smash{{\SetFigFont{7}{8.4}{\rmdefault}{\mddefault}{\updefault}{\color[rgb]{0,0,0}$U_1$}%
}}}}
\put(14851,929){\makebox(0,0)[lb]{\smash{{\SetFigFont{7}{8.4}{\rmdefault}{\mddefault}{\updefault}{\color[rgb]{0,0,0}$U_4$}%
}}}}
\put(14491,1829){\makebox(0,0)[lb]{\smash{{\SetFigFont{7}{8.4}{\rmdefault}{\mddefault}{\updefault}{\color[rgb]{0,0,0}$L_3$}%
}}}}
\put(14581,-61){\makebox(0,0)[lb]{\smash{{\SetFigFont{7}{8.4}{\rmdefault}{\mddefault}{\updefault}{\color[rgb]{0,0,0}$R_3$}%
}}}}
\put(13771,1829){\makebox(0,0)[lb]{\smash{{\SetFigFont{7}{8.4}{\rmdefault}{\mddefault}{\updefault}{\color[rgb]{0,0,0}$L_2$}%
}}}}
\put(13771,-61){\makebox(0,0)[lb]{\smash{{\SetFigFont{7}{8.4}{\rmdefault}{\mddefault}{\updefault}{\color[rgb]{0,0,0}$R_2$}%
}}}}
\put(13411,929){\makebox(0,0)[lb]{\smash{{\SetFigFont{7}{8.4}{\rmdefault}{\mddefault}{\updefault}{\color[rgb]{0,0,0}$U_1$}%
}}}}
\put(14401,-1636){\makebox(0,0)[lb]{\smash{{\SetFigFont{7}{8.4}{\rmdefault}{\mddefault}{\updefault}{\color[rgb]{0,0,0}$L_3$}%
}}}}
\put(13726,-1636){\makebox(0,0)[lb]{\smash{{\SetFigFont{7}{8.4}{\rmdefault}{\mddefault}{\updefault}{\color[rgb]{0,0,0}$L_2$}%
}}}}
\put(14851,-2311){\makebox(0,0)[lb]{\smash{{\SetFigFont{7}{8.4}{\rmdefault}{\mddefault}{\updefault}{\color[rgb]{0,0,0}$U_4$}%
}}}}
\put(14626,-2986){\makebox(0,0)[lb]{\smash{{\SetFigFont{7}{8.4}{\rmdefault}{\mddefault}{\updefault}{\color[rgb]{0,0,0}$R_3$}%
}}}}
\put(13501,-2311){\makebox(0,0)[lb]{\smash{{\SetFigFont{7}{8.4}{\rmdefault}{\mddefault}{\updefault}{\color[rgb]{0,0,0}$U_1$}%
}}}}
\put(13726,-2986){\makebox(0,0)[lb]{\smash{{\SetFigFont{7}{8.4}{\rmdefault}{\mddefault}{\updefault}{\color[rgb]{0,0,0}$R_2$}%
}}}}
\put(20701,929){\makebox(0,0)[lb]{\smash{{\SetFigFont{7}{8.4}{\rmdefault}{\mddefault}{\updefault}{\color[rgb]{0,0,0}$U_4$}%
}}}}
\put(20341,1829){\makebox(0,0)[lb]{\smash{{\SetFigFont{7}{8.4}{\rmdefault}{\mddefault}{\updefault}{\color[rgb]{0,0,0}$L_3$}%
}}}}
\put(20431,-61){\makebox(0,0)[lb]{\smash{{\SetFigFont{7}{8.4}{\rmdefault}{\mddefault}{\updefault}{\color[rgb]{0,0,0}$R_3$}%
}}}}
\put(19621,1829){\makebox(0,0)[lb]{\smash{{\SetFigFont{7}{8.4}{\rmdefault}{\mddefault}{\updefault}{\color[rgb]{0,0,0}$L_2$}%
}}}}
\put(19621,-61){\makebox(0,0)[lb]{\smash{{\SetFigFont{7}{8.4}{\rmdefault}{\mddefault}{\updefault}{\color[rgb]{0,0,0}$R_2$}%
}}}}
\put(19261,929){\makebox(0,0)[lb]{\smash{{\SetFigFont{7}{8.4}{\rmdefault}{\mddefault}{\updefault}{\color[rgb]{0,0,0}$U_1$}%
}}}}
\put(20251,-1636){\makebox(0,0)[lb]{\smash{{\SetFigFont{7}{8.4}{\rmdefault}{\mddefault}{\updefault}{\color[rgb]{0,0,0}$L_3$}%
}}}}
\put(19576,-1636){\makebox(0,0)[lb]{\smash{{\SetFigFont{7}{8.4}{\rmdefault}{\mddefault}{\updefault}{\color[rgb]{0,0,0}$L_2$}%
}}}}
\put(20701,-2311){\makebox(0,0)[lb]{\smash{{\SetFigFont{7}{8.4}{\rmdefault}{\mddefault}{\updefault}{\color[rgb]{0,0,0}$U_4$}%
}}}}
\put(20476,-2986){\makebox(0,0)[lb]{\smash{{\SetFigFont{7}{8.4}{\rmdefault}{\mddefault}{\updefault}{\color[rgb]{0,0,0}$R_3$}%
}}}}
\put(19351,-2311){\makebox(0,0)[lb]{\smash{{\SetFigFont{7}{8.4}{\rmdefault}{\mddefault}{\updefault}{\color[rgb]{0,0,0}$U_1$}%
}}}}
\put(19576,-2986){\makebox(0,0)[lb]{\smash{{\SetFigFont{7}{8.4}{\rmdefault}{\mddefault}{\updefault}{\color[rgb]{0,0,0}$R_2$}%
}}}}
\put(11701,-961){\makebox(0,0)[lb]{\smash{{\SetFigFont{7}{8.4}{\rmdefault}{\mddefault}{\updefault}{\color[rgb]{0,0,0}$U_1$}%
}}}}
\end{picture}%

%% file: Mu0.pstex_t
\begin{picture}(0,0)%
\includegraphics{Mu0.pstex}%
\end{picture}%
\setlength{\unitlength}{1243sp}%
\begingroup\makeatletter\ifx\SetFigFont\undefined%
\gdef\SetFigFont#1#2#3#4#5{%
  \reset@font\fontsize{#1}{#2pt}%
  \fontfamily{#3}\fontseries{#4}\fontshape{#5}%
  \selectfont}%
\fi\endgroup%
\begin{picture}(10586,8056)(12623,-5159)
\put(16201,-691){\makebox(0,0)[lb]{\smash{{\SetFigFont{7}{8.4}{\rmdefault}{\mddefault}{\updefault}{\color[rgb]{1,0,0}$S$}%
}}}}
\put(13726,-2986){\makebox(0,0)[lb]{\smash{{\SetFigFont{7}{8.4}{\rmdefault}{\mddefault}{\updefault}{\color[rgb]{0,0,0}$R_2$}%
}}}}
\put(22051,929){\makebox(0,0)[lb]{\smash{{\SetFigFont{7}{8.4}{\rmdefault}{\mddefault}{\updefault}{\color[rgb]{0,0,0}$U_4$}%
}}}}
\put(21691,1829){\makebox(0,0)[lb]{\smash{{\SetFigFont{7}{8.4}{\rmdefault}{\mddefault}{\updefault}{\color[rgb]{0,0,0}$L_3$}%
}}}}
\put(21781,-61){\makebox(0,0)[lb]{\smash{{\SetFigFont{7}{8.4}{\rmdefault}{\mddefault}{\updefault}{\color[rgb]{0,0,0}$R_3$}%
}}}}
\put(20971,1829){\makebox(0,0)[lb]{\smash{{\SetFigFont{7}{8.4}{\rmdefault}{\mddefault}{\updefault}{\color[rgb]{0,0,0}$L_2$}%
}}}}
\put(20971,-61){\makebox(0,0)[lb]{\smash{{\SetFigFont{7}{8.4}{\rmdefault}{\mddefault}{\updefault}{\color[rgb]{0,0,0}$R_2$}%
}}}}
\put(20611,929){\makebox(0,0)[lb]{\smash{{\SetFigFont{7}{8.4}{\rmdefault}{\mddefault}{\updefault}{\color[rgb]{0,0,0}$U_1$}%
}}}}
\put(21601,-1636){\makebox(0,0)[lb]{\smash{{\SetFigFont{7}{8.4}{\rmdefault}{\mddefault}{\updefault}{\color[rgb]{0,0,0}$L_3$}%
}}}}
\put(20926,-1636){\makebox(0,0)[lb]{\smash{{\SetFigFont{7}{8.4}{\rmdefault}{\mddefault}{\updefault}{\color[rgb]{0,0,0}$L_2$}%
}}}}
\put(22051,-2311){\makebox(0,0)[lb]{\smash{{\SetFigFont{7}{8.4}{\rmdefault}{\mddefault}{\updefault}{\color[rgb]{0,0,0}$U_4$}%
}}}}
\put(21826,-2986){\makebox(0,0)[lb]{\smash{{\SetFigFont{7}{8.4}{\rmdefault}{\mddefault}{\updefault}{\color[rgb]{0,0,0}$R_3$}%
}}}}
\put(20701,-2311){\makebox(0,0)[lb]{\smash{{\SetFigFont{7}{8.4}{\rmdefault}{\mddefault}{\updefault}{\color[rgb]{0,0,0}$U_1$}%
}}}}
\put(20926,-2986){\makebox(0,0)[lb]{\smash{{\SetFigFont{7}{8.4}{\rmdefault}{\mddefault}{\updefault}{\color[rgb]{0,0,0}$R_2$}%
}}}}
\put(19351,-3211){\makebox(0,0)[lb]{\smash{{\SetFigFont{7}{8.4}{\rmdefault}{\mddefault}{\updefault}{\color[rgb]{1,0,0}$S$}%
}}}}
\put(14851,929){\makebox(0,0)[lb]{\smash{{\SetFigFont{7}{8.4}{\rmdefault}{\mddefault}{\updefault}{\color[rgb]{0,0,0}$U_4$}%
}}}}
\put(14491,1829){\makebox(0,0)[lb]{\smash{{\SetFigFont{7}{8.4}{\rmdefault}{\mddefault}{\updefault}{\color[rgb]{0,0,0}$L_3$}%
}}}}
\put(14581,-61){\makebox(0,0)[lb]{\smash{{\SetFigFont{7}{8.4}{\rmdefault}{\mddefault}{\updefault}{\color[rgb]{0,0,0}$R_3$}%
}}}}
\put(13771,1829){\makebox(0,0)[lb]{\smash{{\SetFigFont{7}{8.4}{\rmdefault}{\mddefault}{\updefault}{\color[rgb]{0,0,0}$L_2$}%
}}}}
\put(13771,-61){\makebox(0,0)[lb]{\smash{{\SetFigFont{7}{8.4}{\rmdefault}{\mddefault}{\updefault}{\color[rgb]{0,0,0}$R_2$}%
}}}}
\put(13411,929){\makebox(0,0)[lb]{\smash{{\SetFigFont{7}{8.4}{\rmdefault}{\mddefault}{\updefault}{\color[rgb]{0,0,0}$U_1$}%
}}}}
\put(14401,-1636){\makebox(0,0)[lb]{\smash{{\SetFigFont{7}{8.4}{\rmdefault}{\mddefault}{\updefault}{\color[rgb]{0,0,0}$L_3$}%
}}}}
\put(13726,-1636){\makebox(0,0)[lb]{\smash{{\SetFigFont{7}{8.4}{\rmdefault}{\mddefault}{\updefault}{\color[rgb]{0,0,0}$L_2$}%
}}}}
\put(14851,-2311){\makebox(0,0)[lb]{\smash{{\SetFigFont{7}{8.4}{\rmdefault}{\mddefault}{\updefault}{\color[rgb]{0,0,0}$U_4$}%
}}}}
\put(14626,-2986){\makebox(0,0)[lb]{\smash{{\SetFigFont{7}{8.4}{\rmdefault}{\mddefault}{\updefault}{\color[rgb]{0,0,0}$R_3$}%
}}}}
\put(13501,-2311){\makebox(0,0)[lb]{\smash{{\SetFigFont{7}{8.4}{\rmdefault}{\mddefault}{\updefault}{\color[rgb]{0,0,0}$U_1$}%
}}}}
\end{picture}%

%% file: Mu0cancel.pstex_t
\begin{picture}(0,0)%
\includegraphics{Mu0cancel.pstex}%
\end{picture}%
\setlength{\unitlength}{1243sp}%
\begingroup\makeatletter\ifx\SetFigFont\undefined%
\gdef\SetFigFont#1#2#3#4#5{%
  \reset@font\fontsize{#1}{#2pt}%
  \fontfamily{#3}\fontseries{#4}\fontshape{#5}%
  \selectfont}%
\fi\endgroup%
\begin{picture}(10157,8056)(12623,-5159)
\put(16201,-691){\makebox(0,0)[lb]{\smash{{\SetFigFont{7}{8.4}{\rmdefault}{\mddefault}{\updefault}{\color[rgb]{1,0,0}$S$}%
}}}}
\put(13726,-2986){\makebox(0,0)[lb]{\smash{{\SetFigFont{7}{8.4}{\rmdefault}{\mddefault}{\updefault}{\color[rgb]{0,0,0}$R_2$}%
}}}}
\put(21601,929){\makebox(0,0)[lb]{\smash{{\SetFigFont{7}{8.4}{\rmdefault}{\mddefault}{\updefault}{\color[rgb]{0,0,0}$U_4$}%
}}}}
\put(21241,1829){\makebox(0,0)[lb]{\smash{{\SetFigFont{7}{8.4}{\rmdefault}{\mddefault}{\updefault}{\color[rgb]{0,0,0}$L_3$}%
}}}}
\put(21331,-61){\makebox(0,0)[lb]{\smash{{\SetFigFont{7}{8.4}{\rmdefault}{\mddefault}{\updefault}{\color[rgb]{0,0,0}$R_3$}%
}}}}
\put(20521,1829){\makebox(0,0)[lb]{\smash{{\SetFigFont{7}{8.4}{\rmdefault}{\mddefault}{\updefault}{\color[rgb]{0,0,0}$L_2$}%
}}}}
\put(20521,-61){\makebox(0,0)[lb]{\smash{{\SetFigFont{7}{8.4}{\rmdefault}{\mddefault}{\updefault}{\color[rgb]{0,0,0}$R_2$}%
}}}}
\put(20161,929){\makebox(0,0)[lb]{\smash{{\SetFigFont{7}{8.4}{\rmdefault}{\mddefault}{\updefault}{\color[rgb]{0,0,0}$U_1$}%
}}}}
\put(21151,-1636){\makebox(0,0)[lb]{\smash{{\SetFigFont{7}{8.4}{\rmdefault}{\mddefault}{\updefault}{\color[rgb]{0,0,0}$L_3$}%
}}}}
\put(20476,-1636){\makebox(0,0)[lb]{\smash{{\SetFigFont{7}{8.4}{\rmdefault}{\mddefault}{\updefault}{\color[rgb]{0,0,0}$L_2$}%
}}}}
\put(21601,-2311){\makebox(0,0)[lb]{\smash{{\SetFigFont{7}{8.4}{\rmdefault}{\mddefault}{\updefault}{\color[rgb]{0,0,0}$U_4$}%
}}}}
\put(21376,-2986){\makebox(0,0)[lb]{\smash{{\SetFigFont{7}{8.4}{\rmdefault}{\mddefault}{\updefault}{\color[rgb]{0,0,0}$R_3$}%
}}}}
\put(20251,-2311){\makebox(0,0)[lb]{\smash{{\SetFigFont{7}{8.4}{\rmdefault}{\mddefault}{\updefault}{\color[rgb]{0,0,0}$U_1$}%
}}}}
\put(20476,-2986){\makebox(0,0)[lb]{\smash{{\SetFigFont{7}{8.4}{\rmdefault}{\mddefault}{\updefault}{\color[rgb]{0,0,0}$R_2$}%
}}}}
\put(14851,929){\makebox(0,0)[lb]{\smash{{\SetFigFont{7}{8.4}{\rmdefault}{\mddefault}{\updefault}{\color[rgb]{0,0,0}$U_4$}%
}}}}
\put(14491,1829){\makebox(0,0)[lb]{\smash{{\SetFigFont{7}{8.4}{\rmdefault}{\mddefault}{\updefault}{\color[rgb]{0,0,0}$L_3$}%
}}}}
\put(14581,-61){\makebox(0,0)[lb]{\smash{{\SetFigFont{7}{8.4}{\rmdefault}{\mddefault}{\updefault}{\color[rgb]{0,0,0}$R_3$}%
}}}}
\put(13771,1829){\makebox(0,0)[lb]{\smash{{\SetFigFont{7}{8.4}{\rmdefault}{\mddefault}{\updefault}{\color[rgb]{0,0,0}$L_2$}%
}}}}
\put(13771,-61){\makebox(0,0)[lb]{\smash{{\SetFigFont{7}{8.4}{\rmdefault}{\mddefault}{\updefault}{\color[rgb]{0,0,0}$R_2$}%
}}}}
\put(13411,929){\makebox(0,0)[lb]{\smash{{\SetFigFont{7}{8.4}{\rmdefault}{\mddefault}{\updefault}{\color[rgb]{0,0,0}$U_1$}%
}}}}
\put(14401,-1636){\makebox(0,0)[lb]{\smash{{\SetFigFont{7}{8.4}{\rmdefault}{\mddefault}{\updefault}{\color[rgb]{0,0,0}$L_3$}%
}}}}
\put(13726,-1636){\makebox(0,0)[lb]{\smash{{\SetFigFont{7}{8.4}{\rmdefault}{\mddefault}{\updefault}{\color[rgb]{0,0,0}$L_2$}%
}}}}
\put(14851,-2311){\makebox(0,0)[lb]{\smash{{\SetFigFont{7}{8.4}{\rmdefault}{\mddefault}{\updefault}{\color[rgb]{0,0,0}$U_4$}%
}}}}
\put(14626,-2986){\makebox(0,0)[lb]{\smash{{\SetFigFont{7}{8.4}{\rmdefault}{\mddefault}{\updefault}{\color[rgb]{0,0,0}$R_3$}%
}}}}
\put(13501,-2311){\makebox(0,0)[lb]{\smash{{\SetFigFont{7}{8.4}{\rmdefault}{\mddefault}{\updefault}{\color[rgb]{0,0,0}$U_1$}%
}}}}
\end{picture}%

%% file: C0pC0m.pstex_t
\begin{picture}(0,0)%
\includegraphics{C0pC0m.pstex}%
\end{picture}%
\setlength{\unitlength}{1243sp}%
\begingroup\makeatletter\ifx\SetFigFont\undefined%
\gdef\SetFigFont#1#2#3#4#5{%
  \reset@font\fontsize{#1}{#2pt}%
  \fontfamily{#3}\fontseries{#4}\fontshape{#5}%
  \selectfont}%
\fi\endgroup%
\begin{picture}(11298,8056)(11911,-5159)
\put(22051,929){\makebox(0,0)[lb]{\smash{{\SetFigFont{7}{8.4}{\rmdefault}{\mddefault}{\updefault}{\color[rgb]{0,0,0}$U_4$}%
}}}}
\put(21691,1829){\makebox(0,0)[lb]{\smash{{\SetFigFont{7}{8.4}{\rmdefault}{\mddefault}{\updefault}{\color[rgb]{0,0,0}$L_3$}%
}}}}
\put(21781,-61){\makebox(0,0)[lb]{\smash{{\SetFigFont{7}{8.4}{\rmdefault}{\mddefault}{\updefault}{\color[rgb]{0,0,0}$R_3$}%
}}}}
\put(20971,1829){\makebox(0,0)[lb]{\smash{{\SetFigFont{7}{8.4}{\rmdefault}{\mddefault}{\updefault}{\color[rgb]{0,0,0}$L_2$}%
}}}}
\put(20971,-61){\makebox(0,0)[lb]{\smash{{\SetFigFont{7}{8.4}{\rmdefault}{\mddefault}{\updefault}{\color[rgb]{0,0,0}$R_2$}%
}}}}
\put(20611,929){\makebox(0,0)[lb]{\smash{{\SetFigFont{7}{8.4}{\rmdefault}{\mddefault}{\updefault}{\color[rgb]{0,0,0}$U_1$}%
}}}}
\put(21601,-1636){\makebox(0,0)[lb]{\smash{{\SetFigFont{7}{8.4}{\rmdefault}{\mddefault}{\updefault}{\color[rgb]{0,0,0}$L_3$}%
}}}}
\put(20926,-1636){\makebox(0,0)[lb]{\smash{{\SetFigFont{7}{8.4}{\rmdefault}{\mddefault}{\updefault}{\color[rgb]{0,0,0}$L_2$}%
}}}}
\put(22051,-2311){\makebox(0,0)[lb]{\smash{{\SetFigFont{7}{8.4}{\rmdefault}{\mddefault}{\updefault}{\color[rgb]{0,0,0}$U_4$}%
}}}}
\put(21826,-2986){\makebox(0,0)[lb]{\smash{{\SetFigFont{7}{8.4}{\rmdefault}{\mddefault}{\updefault}{\color[rgb]{0,0,0}$R_3$}%
}}}}
\put(20701,-2311){\makebox(0,0)[lb]{\smash{{\SetFigFont{7}{8.4}{\rmdefault}{\mddefault}{\updefault}{\color[rgb]{0,0,0}$U_1$}%
}}}}
\put(20926,-2986){\makebox(0,0)[lb]{\smash{{\SetFigFont{7}{8.4}{\rmdefault}{\mddefault}{\updefault}{\color[rgb]{0,0,0}$R_2$}%
}}}}
\put(19351,-3211){\makebox(0,0)[lb]{\smash{{\SetFigFont{7}{8.4}{\rmdefault}{\mddefault}{\updefault}{\color[rgb]{1,0,0}$S$}%
}}}}
\put(14851,929){\makebox(0,0)[lb]{\smash{{\SetFigFont{7}{8.4}{\rmdefault}{\mddefault}{\updefault}{\color[rgb]{0,0,0}$U_4$}%
}}}}
\put(14491,1829){\makebox(0,0)[lb]{\smash{{\SetFigFont{7}{8.4}{\rmdefault}{\mddefault}{\updefault}{\color[rgb]{0,0,0}$L_3$}%
}}}}
\put(14581,-61){\makebox(0,0)[lb]{\smash{{\SetFigFont{7}{8.4}{\rmdefault}{\mddefault}{\updefault}{\color[rgb]{0,0,0}$R_3$}%
}}}}
\put(13771,1829){\makebox(0,0)[lb]{\smash{{\SetFigFont{7}{8.4}{\rmdefault}{\mddefault}{\updefault}{\color[rgb]{0,0,0}$L_2$}%
}}}}
\put(13771,-61){\makebox(0,0)[lb]{\smash{{\SetFigFont{7}{8.4}{\rmdefault}{\mddefault}{\updefault}{\color[rgb]{0,0,0}$R_2$}%
}}}}
\put(13411,929){\makebox(0,0)[lb]{\smash{{\SetFigFont{7}{8.4}{\rmdefault}{\mddefault}{\updefault}{\color[rgb]{0,0,0}$U_1$}%
}}}}
\put(14401,-1636){\makebox(0,0)[lb]{\smash{{\SetFigFont{7}{8.4}{\rmdefault}{\mddefault}{\updefault}{\color[rgb]{0,0,0}$L_3$}%
}}}}
\put(13726,-1636){\makebox(0,0)[lb]{\smash{{\SetFigFont{7}{8.4}{\rmdefault}{\mddefault}{\updefault}{\color[rgb]{0,0,0}$L_2$}%
}}}}
\put(14851,-2311){\makebox(0,0)[lb]{\smash{{\SetFigFont{7}{8.4}{\rmdefault}{\mddefault}{\updefault}{\color[rgb]{0,0,0}$U_4$}%
}}}}
\put(14626,-2986){\makebox(0,0)[lb]{\smash{{\SetFigFont{7}{8.4}{\rmdefault}{\mddefault}{\updefault}{\color[rgb]{0,0,0}$R_3$}%
}}}}
\put(13501,-2311){\makebox(0,0)[lb]{\smash{{\SetFigFont{7}{8.4}{\rmdefault}{\mddefault}{\updefault}{\color[rgb]{0,0,0}$U_1$}%
}}}}
\put(13726,-2986){\makebox(0,0)[lb]{\smash{{\SetFigFont{7}{8.4}{\rmdefault}{\mddefault}{\updefault}{\color[rgb]{0,0,0}$R_2$}%
}}}}
\put(11926,-3391){\makebox(0,0)[lb]{\smash{{\SetFigFont{7}{8.4}{\rmdefault}{\mddefault}{\updefault}{\color[rgb]{1,0,0}$S$}%
}}}}
\end{picture}%

%% file: L0pR0m.pstex_t
\begin{picture}(0,0)%
\includegraphics{L0pR0m.pstex}%
\end{picture}%
\setlength{\unitlength}{1243sp}%
\begingroup\makeatletter\ifx\SetFigFont\undefined%
\gdef\SetFigFont#1#2#3#4#5{%
  \reset@font\fontsize{#1}{#2pt}%
  \fontfamily{#3}\fontseries{#4}\fontshape{#5}%
  \selectfont}%
\fi\endgroup%
\begin{picture}(10416,8652)(12793,-5755)
\put(22051,929){\makebox(0,0)[lb]{\smash{{\SetFigFont{7}{8.4}{\rmdefault}{\mddefault}{\updefault}{\color[rgb]{0,0,0}$U_4$}%
}}}}
\put(21691,1829){\makebox(0,0)[lb]{\smash{{\SetFigFont{7}{8.4}{\rmdefault}{\mddefault}{\updefault}{\color[rgb]{0,0,0}$L_3$}%
}}}}
\put(21781,-61){\makebox(0,0)[lb]{\smash{{\SetFigFont{7}{8.4}{\rmdefault}{\mddefault}{\updefault}{\color[rgb]{0,0,0}$R_3$}%
}}}}
\put(20971,1829){\makebox(0,0)[lb]{\smash{{\SetFigFont{7}{8.4}{\rmdefault}{\mddefault}{\updefault}{\color[rgb]{0,0,0}$L_2$}%
}}}}
\put(20971,-61){\makebox(0,0)[lb]{\smash{{\SetFigFont{7}{8.4}{\rmdefault}{\mddefault}{\updefault}{\color[rgb]{0,0,0}$R_2$}%
}}}}
\put(20611,929){\makebox(0,0)[lb]{\smash{{\SetFigFont{7}{8.4}{\rmdefault}{\mddefault}{\updefault}{\color[rgb]{0,0,0}$U_1$}%
}}}}
\put(21601,-1636){\makebox(0,0)[lb]{\smash{{\SetFigFont{7}{8.4}{\rmdefault}{\mddefault}{\updefault}{\color[rgb]{0,0,0}$L_3$}%
}}}}
\put(20926,-1636){\makebox(0,0)[lb]{\smash{{\SetFigFont{7}{8.4}{\rmdefault}{\mddefault}{\updefault}{\color[rgb]{0,0,0}$L_2$}%
}}}}
\put(22051,-2311){\makebox(0,0)[lb]{\smash{{\SetFigFont{7}{8.4}{\rmdefault}{\mddefault}{\updefault}{\color[rgb]{0,0,0}$U_4$}%
}}}}
\put(21826,-2986){\makebox(0,0)[lb]{\smash{{\SetFigFont{7}{8.4}{\rmdefault}{\mddefault}{\updefault}{\color[rgb]{0,0,0}$R_3$}%
}}}}
\put(20701,-2311){\makebox(0,0)[lb]{\smash{{\SetFigFont{7}{8.4}{\rmdefault}{\mddefault}{\updefault}{\color[rgb]{0,0,0}$U_1$}%
}}}}
\put(20926,-2986){\makebox(0,0)[lb]{\smash{{\SetFigFont{7}{8.4}{\rmdefault}{\mddefault}{\updefault}{\color[rgb]{0,0,0}$R_2$}%
}}}}
\put(14851,929){\makebox(0,0)[lb]{\smash{{\SetFigFont{7}{8.4}{\rmdefault}{\mddefault}{\updefault}{\color[rgb]{0,0,0}$U_4$}%
}}}}
\put(14491,1829){\makebox(0,0)[lb]{\smash{{\SetFigFont{7}{8.4}{\rmdefault}{\mddefault}{\updefault}{\color[rgb]{0,0,0}$L_3$}%
}}}}
\put(14581,-61){\makebox(0,0)[lb]{\smash{{\SetFigFont{7}{8.4}{\rmdefault}{\mddefault}{\updefault}{\color[rgb]{0,0,0}$R_3$}%
}}}}
\put(13771,1829){\makebox(0,0)[lb]{\smash{{\SetFigFont{7}{8.4}{\rmdefault}{\mddefault}{\updefault}{\color[rgb]{0,0,0}$L_2$}%
}}}}
\put(13771,-61){\makebox(0,0)[lb]{\smash{{\SetFigFont{7}{8.4}{\rmdefault}{\mddefault}{\updefault}{\color[rgb]{0,0,0}$R_2$}%
}}}}
\put(13411,929){\makebox(0,0)[lb]{\smash{{\SetFigFont{7}{8.4}{\rmdefault}{\mddefault}{\updefault}{\color[rgb]{0,0,0}$U_1$}%
}}}}
\put(14401,-1636){\makebox(0,0)[lb]{\smash{{\SetFigFont{7}{8.4}{\rmdefault}{\mddefault}{\updefault}{\color[rgb]{0,0,0}$L_3$}%
}}}}
\put(13726,-1636){\makebox(0,0)[lb]{\smash{{\SetFigFont{7}{8.4}{\rmdefault}{\mddefault}{\updefault}{\color[rgb]{0,0,0}$L_2$}%
}}}}
\put(14851,-2311){\makebox(0,0)[lb]{\smash{{\SetFigFont{7}{8.4}{\rmdefault}{\mddefault}{\updefault}{\color[rgb]{0,0,0}$U_4$}%
}}}}
\put(14626,-2986){\makebox(0,0)[lb]{\smash{{\SetFigFont{7}{8.4}{\rmdefault}{\mddefault}{\updefault}{\color[rgb]{0,0,0}$R_3$}%
}}}}
\put(13501,-2311){\makebox(0,0)[lb]{\smash{{\SetFigFont{7}{8.4}{\rmdefault}{\mddefault}{\updefault}{\color[rgb]{0,0,0}$U_1$}%
}}}}
\put(13726,-2986){\makebox(0,0)[lb]{\smash{{\SetFigFont{7}{8.4}{\rmdefault}{\mddefault}{\updefault}{\color[rgb]{0,0,0}$R_2$}%
}}}}
\put(14536,-5641){\makebox(0,0)[lb]{\smash{{\SetFigFont{7}{8.4}{\rmdefault}{\mddefault}{\updefault}{\color[rgb]{1,0,0}$S$}%
}}}}
\put(21601,-5641){\makebox(0,0)[lb]{\smash{{\SetFigFont{7}{8.4}{\rmdefault}{\mddefault}{\updefault}{\color[rgb]{1,0,0}$S$}%
}}}}
\put(14086,-3931){\makebox(0,0)[lb]{\smash{{\SetFigFont{7}{8.4}{\rmdefault}{\mddefault}{\updefault}{\color[rgb]{0,0,0}$L_2$}%
}}}}
\put(20746,-3886){\makebox(0,0)[lb]{\smash{{\SetFigFont{7}{8.4}{\rmdefault}{\mddefault}{\updefault}{\color[rgb]{0,0,0}$L_2$}%
}}}}
\end{picture}%

%% file: L0mEx.pstex_t
\begin{picture}(0,0)%
\includegraphics{L0mEx.pstex}%
\end{picture}%
\setlength{\unitlength}{829sp}%
\begingroup\makeatletter\ifx\SetFigFont\undefined%
\gdef\SetFigFont#1#2#3#4#5{%
  \reset@font\fontsize{#1}{#2pt}%
  \fontfamily{#3}\fontseries{#4}\fontshape{#5}%
  \selectfont}%
\fi\endgroup%
\begin{picture}(22537,16534)(-3666,-18099)
\put(18856,-10096){\makebox(0,0)[rb]{\smash{{\SetFigFont{5}{6.0}{\rmdefault}{\mddefault}{\updefault}{\color[rgb]{0,0,0}$L_2$}%
}}}}
\put(18766,-16351){\makebox(0,0)[rb]{\smash{{\SetFigFont{5}{6.0}{\rmdefault}{\mddefault}{\updefault}{\color[rgb]{0,0,0}$L_2$}%
}}}}
\put(1351,-16036){\makebox(0,0)[rb]{\smash{{\SetFigFont{5}{6.0}{\rmdefault}{\mddefault}{\updefault}{\color[rgb]{0,0,0}$L_2$}%
}}}}
\put(1351,-3436){\makebox(0,0)[rb]{\smash{{\SetFigFont{5}{6.0}{\rmdefault}{\mddefault}{\updefault}{\color[rgb]{0,0,0}$L_2$}%
}}}}
\put(3376,-3841){\makebox(0,0)[rb]{\smash{{\SetFigFont{5}{6.0}{\rmdefault}{\mddefault}{\updefault}{\color[rgb]{0,0,0}$L_2$}%
}}}}
\put(1351,-9736){\makebox(0,0)[rb]{\smash{{\SetFigFont{5}{6.0}{\rmdefault}{\mddefault}{\updefault}{\color[rgb]{0,0,0}$L_2$}%
}}}}
\put(3376,-10141){\makebox(0,0)[rb]{\smash{{\SetFigFont{5}{6.0}{\rmdefault}{\mddefault}{\updefault}{\color[rgb]{0,0,0}$L_2$}%
}}}}
\put(1801,-2896){\makebox(0,0)[rb]{\smash{{\SetFigFont{5}{6.0}{\rmdefault}{\mddefault}{\updefault}{\color[rgb]{0,0,0}$U_3$}%
}}}}
\put(-269,-3256){\makebox(0,0)[rb]{\smash{{\SetFigFont{5}{6.0}{\rmdefault}{\mddefault}{\updefault}{\color[rgb]{0,0,0}$R_2$}%
}}}}
\put(-359,-2581){\makebox(0,0)[rb]{\smash{{\SetFigFont{5}{6.0}{\rmdefault}{\mddefault}{\updefault}{\color[rgb]{0,0,0}$U_3$}%
}}}}
\put(-1169,-2761){\makebox(0,0)[rb]{\smash{{\SetFigFont{5}{6.0}{\rmdefault}{\mddefault}{\updefault}{\color[rgb]{0,0,0}$L_2$}%
}}}}
\put(-1259,-3391){\makebox(0,0)[rb]{\smash{{\SetFigFont{5}{6.0}{\rmdefault}{\mddefault}{\updefault}{\color[rgb]{0,0,0}$U_1$}%
}}}}
\put(2656,-9646){\makebox(0,0)[rb]{\smash{{\SetFigFont{5}{6.0}{\rmdefault}{\mddefault}{\updefault}{\color[rgb]{0,0,0}$R_2$}%
}}}}
\put(1846,-9196){\makebox(0,0)[rb]{\smash{{\SetFigFont{5}{6.0}{\rmdefault}{\mddefault}{\updefault}{\color[rgb]{0,0,0}$U_3$}%
}}}}
\put(2521,-15901){\makebox(0,0)[rb]{\smash{{\SetFigFont{5}{6.0}{\rmdefault}{\mddefault}{\updefault}{\color[rgb]{0,0,0}$R_2$}%
}}}}
\put(1801,-15496){\makebox(0,0)[rb]{\smash{{\SetFigFont{5}{6.0}{\rmdefault}{\mddefault}{\updefault}{\color[rgb]{0,0,0}$U_3$}%
}}}}
\put(2296,-3886){\makebox(0,0)[rb]{\smash{{\SetFigFont{5}{6.0}{\rmdefault}{\mddefault}{\updefault}{\color[rgb]{0,0,0}$U_1$}%
}}}}
\put(2566,-3346){\makebox(0,0)[rb]{\smash{{\SetFigFont{5}{6.0}{\rmdefault}{\mddefault}{\updefault}{\color[rgb]{0,0,0}$R_2$}%
}}}}
\put(2251,-10231){\makebox(0,0)[rb]{\smash{{\SetFigFont{5}{6.0}{\rmdefault}{\mddefault}{\updefault}{\color[rgb]{0,0,0}$U_1$}%
}}}}
\put(2161,-16531){\makebox(0,0)[rb]{\smash{{\SetFigFont{5}{6.0}{\rmdefault}{\mddefault}{\updefault}{\color[rgb]{0,0,0}$U_1$}%
}}}}
\put(3421,-16261){\makebox(0,0)[rb]{\smash{{\SetFigFont{5}{6.0}{\rmdefault}{\mddefault}{\updefault}{\color[rgb]{0,0,0}$L_2$}%
}}}}
\put(4051,-16621){\makebox(0,0)[rb]{\smash{{\SetFigFont{5}{6.0}{\rmdefault}{\mddefault}{\updefault}{\color[rgb]{0,0,0}$R_2$}%
}}}}
\put(14131,-16036){\makebox(0,0)[rb]{\smash{{\SetFigFont{5}{6.0}{\rmdefault}{\mddefault}{\updefault}{\color[rgb]{0,0,0}$L_2$}%
}}}}
\put(13231,-16981){\makebox(0,0)[rb]{\smash{{\SetFigFont{5}{6.0}{\rmdefault}{\mddefault}{\updefault}{\color[rgb]{0,0,0}$R_2$}%
}}}}
\put(12016,-3346){\makebox(0,0)[rb]{\smash{{\SetFigFont{5}{6.0}{\rmdefault}{\mddefault}{\updefault}{\color[rgb]{0,0,0}$R_2$}%
}}}}
\put(12331,-4021){\makebox(0,0)[rb]{\smash{{\SetFigFont{5}{6.0}{\rmdefault}{\mddefault}{\updefault}{\color[rgb]{0,0,0}$L_2$}%
}}}}
\put(14131,-3436){\makebox(0,0)[rb]{\smash{{\SetFigFont{5}{6.0}{\rmdefault}{\mddefault}{\updefault}{\color[rgb]{0,0,0}$L_2$}%
}}}}
\put(14131,-9736){\makebox(0,0)[rb]{\smash{{\SetFigFont{5}{6.0}{\rmdefault}{\mddefault}{\updefault}{\color[rgb]{0,0,0}$L_2$}%
}}}}
\put(15121,-3256){\makebox(0,0)[rb]{\smash{{\SetFigFont{5}{6.0}{\rmdefault}{\mddefault}{\updefault}{\color[rgb]{0,0,0}$R_2$}%
}}}}
\put(14536,-2851){\makebox(0,0)[rb]{\smash{{\SetFigFont{5}{6.0}{\rmdefault}{\mddefault}{\updefault}{\color[rgb]{0,0,0}$U_3$}%
}}}}
\put(12286,-2716){\makebox(0,0)[rb]{\smash{{\SetFigFont{5}{6.0}{\rmdefault}{\mddefault}{\updefault}{\color[rgb]{0,0,0}$U_3$}%
}}}}
\put(11206,-2671){\makebox(0,0)[rb]{\smash{{\SetFigFont{5}{6.0}{\rmdefault}{\mddefault}{\updefault}{\color[rgb]{0,0,0}$L_2$}%
}}}}
\put(11071,-3256){\makebox(0,0)[rb]{\smash{{\SetFigFont{5}{6.0}{\rmdefault}{\mddefault}{\updefault}{\color[rgb]{0,0,0}$U_1$}%
}}}}
\put(15256,-9556){\makebox(0,0)[rb]{\smash{{\SetFigFont{5}{6.0}{\rmdefault}{\mddefault}{\updefault}{\color[rgb]{0,0,0}$R_2$}%
}}}}
\put(14536,-9151){\makebox(0,0)[rb]{\smash{{\SetFigFont{5}{6.0}{\rmdefault}{\mddefault}{\updefault}{\color[rgb]{0,0,0}$U_3$}%
}}}}
\put(15166,-15856){\makebox(0,0)[rb]{\smash{{\SetFigFont{5}{6.0}{\rmdefault}{\mddefault}{\updefault}{\color[rgb]{0,0,0}$R_2$}%
}}}}
\put(14536,-15451){\makebox(0,0)[rb]{\smash{{\SetFigFont{5}{6.0}{\rmdefault}{\mddefault}{\updefault}{\color[rgb]{0,0,0}$U_3$}%
}}}}
\put(14896,-3886){\makebox(0,0)[rb]{\smash{{\SetFigFont{5}{6.0}{\rmdefault}{\mddefault}{\updefault}{\color[rgb]{0,0,0}$U_1$}%
}}}}
\put(14896,-10186){\makebox(0,0)[rb]{\smash{{\SetFigFont{5}{6.0}{\rmdefault}{\mddefault}{\updefault}{\color[rgb]{0,0,0}$U_1$}%
}}}}
\put(14941,-16486){\makebox(0,0)[rb]{\smash{{\SetFigFont{5}{6.0}{\rmdefault}{\mddefault}{\updefault}{\color[rgb]{0,0,0}$U_1$}%
}}}}
\end{picture}%

%% file: signs.tex
\section{The $\Ainfty$ relations with signs}
\label{sec:Signs}

Analogous to~\cite[Section~7]{TorusAlg}, lifting the constructions of
this paper to $\Z$ coefficients is a fairly straightforward matter.

As a preliminary step, 
we recall the sign conventions on weighted $\Ainfty$ algebras,
following~\cite[Section~7]{TorusAlg}, generalizing~\cite{Keller}.

To make sense of these sign conventions, one must start with a
$\Zmod{2}$ grading on the underlying algebra, which in our case is
$\Clg[t]$. We endow it with a trivial $\Zmod{2}$ grading,
supported entirely in grading $0$.

The sign convention on a
weighted $\Ainfty$ algebras states:
\begin{equation}
  \label{eq:WeightedAinf}
  \sum_{\substack{n=r+s+t\\w=u+v}} (-1)^{r+st} \mu^u_{r+1+t}\circ(\Id^{\otimes r} \otimes \mu^v_s \otimes \Id^{\otimes t})=0.
\end{equation}
(In general, the composition of homomorphisms, such as the ones
appearing above, satisfy a Leibinz rule. For our purposes,
we can suppress this, as our algebra is supported in degree $0$.)

\begin{thm}
  \label{thm:WeightedAinftyAlgebraZ}
  The operations $\{\mu^{\vec{w}}_n\}$ give $\ClgZ[t]$ the structure
  of a unital, weighted $\Ainfty$ algebra (over $\Z$).
\end{thm}

\begin{proof}
  We note that the pairwise cancellations from the proof of
  Theorem~\ref{thm:WeightedAinftyAlgebra} occur with opposite signs.

  To this end, it helps to notice Lemma~\ref{lem:ActionsAreEven}.
  
  Consider the cancellation from Lemma~\ref{lem:CancelC2}.  Let
  $(r_1,s_1,t_1)$ be the integers in Equation~\eqref{eq:WeightedAinf}
  for the term of type $(C2)$ and $(r_2,s_2,t_2)$ be integers from the
  cancelling term of type $[C**]$. Our aim is to show that $r_1+s_1
  t_1$ and $r_2+s_2t_2$ have opposite parity.  Since $s_1$ and $s_2$
  are even (Lemma~\ref{lem:ActionsAreEven}), it suffices to show that
  $r_1+r_2$ is odd. 
  The cancelling term can be of the following types:
  \begin{itemize}
  \item Type $[CR*]$, so that
    $r_1<r_2$. In this case, $r_2=r_1+1$.
  \item Type $[CL*]$, when $r_2<r_1$.  In this case, $r_1+1=r_2+n_2$. 
  \item Type $[C0*]$. In that case, $r_2=r_1+1$.
  \end{itemize}
  In all the above cases, $r_1+r_2\equiv 1\pmod{2}$ (again,
  using Lemma~\ref{lem:ActionsAreEven}). This same argument
  also handles the cancellations from Lemma~\ref{lem:CancelL2}.

  Note that for all terms of type $[**+]$, we have $r=0$;
  so these contribute to the $A_\infty$ relation with sign $+1$;
  for all term of type $[**-]$, we have $r=n-1$,
  where $n$ is the number of inputs to the operation at the root
  of the tree. Thus, by Lemma~\ref{lem:ActionsAreEven}, these contribute the
  sign of $-1$.
  It now follows that 
  for  cancellations from Lemmas~\ref{lem:RLmLRp}, 
  \ref{lem:CancelL0pR0m}, and~\ref{lem:LLmRRp},
  the terms in the $A_\infty$ relation all appear 
  with cancelling sign.
\end{proof}